\documentclass{amsproc}
\usepackage[all]{xy}
\usepackage{amscd}
\usepackage{amssymb,amsmath}
\usepackage[usenames]{color}
\usepackage{comment}
\usepackage{framed}
\definecolor{shadecolor}{gray}{0.875}
\specialcomment{personal}{\begin{shaded}}{\end{shaded}}

\newtheorem{theorem}{Theorem}[section]
\newtheorem{lemma}[theorem]{Lemma}
\newtheorem{conjecture}[theorem]{Conjecture}
\newtheorem{corollary}[theorem]{Corollary}
\newtheorem{proposition}[theorem]{Proposition}

\theoremstyle{definition}

\theoremstyle{remark}
\newtheorem{remark}[theorem]{Remark}

\numberwithin{equation}{section}

\renewcommand{\color}[1]{ } 

\renewcommand{\det}{\mathrm{det}}

\renewcommand{\L}{\ifmmode {\mathcal{L}}\else$\mathcal{L}$\ \fi}

\newcommand{\bbC}{\ifmmode {\mathbb{C}}\else$\mathbb{C}$\ \fi}
\newcommand{\bbR}{\ifmmode {\mathbb{R}}\else$\mathbb{R}$\ \fi}

\newcommand{\be}{\begin{equation}}
\newcommand{\ee}{\end{equation}}

\newcommand{\fpbar}{\ifmmode {\overline{\mathbb{F}_p}}\else$\mathbb{F}_p$\ \fi}

\newcommand{\fp}{\ifmmode {\mathbb{F}_p}\else$\mathbb{F}_p$\ \fi}
\newcommand{\zp}{\ifmmode \mathbb{Z}_p\else$\mathbb{Z}_p$\ \fi}
\newcommand{\zpur}{\ifmmode \widehat{\zp^{ur}}\else $\widehat{\zp^{ur}}$\ \fi}
\newcommand{\Tun}{\ifmmode \mathbb{T}_{un}\else$\mathbb{T}_{un}$\ \fi}

\newcommand{\zpMod}{\ifmmode\mbox{$\zp$-Mod}\else$\zp$-Mod \fi}
\newcommand{\Mod}{\ifmmode\mbox{$\Lambda$-Mod}\else$\Lambda$-Mod \fi}
\renewcommand{\mod}{\ifmmode\mbox{$\Lambda$-mod}\else$\Lambda$-mod
\fi}
\newcommand{\La}{\ifmmode\Lambda\else$\Lambda$\fi}

\newcommand{\M}{\ifmmode {\frak M}\else${\frak M}$ \fi}
\newcommand{\m}{\ifmmode {\frak m}\else$\frak m$ \fi}
\newcommand{\mh}{\ifmmode {\frak m}(H)\else${\frak m}(H)$ \fi}
\newcommand{\p}{\ifmmode {\frak p}\else${\frak p}$\ \fi}
\renewcommand{\P}{\ifmmode {\frak P}\else${\frak P}$\ \fi}
\newcommand{\e}{\ifmmode {\mathcal{E}}\else$\mathcal{E}$ \fi}

\newcommand{\T}{\mathbb{ T}}
\newcommand{\V}{\mathbb{ V}}

\newcommand{\G}{\ifmmode {\mathcal{G}}\else${\mathcal{G}}$\ \fi}
\newcommand{\A}{\ifmmode {\mathcal{A}}\else${\mathcal{ A}}$\ \fi}

\newcommand{\Qp}{\ifmmode {{\Bbb Q}_p}\else${\Bbb Q}_p$\ \fi}
\newcommand{\qp}{\ifmmode {\Bbb Q}_p\else${\Bbb Q}_p$\ \fi}
\newcommand{\ql}{\ifmmode {{\Bbb Q}_l}\else${\Bbb Q}_l$\ \fi}
\newcommand{\Q}{\ifmmode {\Bbb Q}\else${\Bbb Q}$\ \fi}
\newcommand{\q}{\ifmmode {\Bbb Q}\else${\Bbb Q}$\ \fi}

\newcommand{\Ind}{\mathrm{Ind}}

\begin{document}

\title{Equivariant epsilon conjecture for $1$-dimensional Lubin-Tate groups}

\author{Dmitriy Izychev}
\author{Otmar Venjakob}

\address{Mathematisches Institut, Universit\"{a}t Heidelberg, Im Neuenheimer Feld 288, 69120 Heidelberg, Germany.}
\email{dizychev@mathi.uni-heidelberg.de}

\email{venjakob@mathi.uni-heidelberg.de}
\urladdr{http://www.mathi.uni-heidelberg.de/\textasciitilde otmar/}

\thanks{Both authors acknowledge support by  the ERC-Starting Grant IWASAWA, the first one in addition by the DAAD, the second one by the DFG-Forschergruppe "Symmetrie, Geometrie und Arithmetik", \today}

\begin{abstract}
In this paper we formulate a conjecture on the relationship between the equivariant $\epsilon$-constants (associated to a local $p$-adic representation $V$ and a finite extension of local fields $L/K$) and local Galois cohomology groups of a Galois stable $\mathbb{Z}_{p}$-lattice $T$ of $V$. We prove the conjecture for $L/K$ being an unramified extension of degree prime to $p$ and $T$ being a $p$-adic Tate module of a one-dimensional Lubin-Tate group defined over $\mathbb{Z}_{p}$ by extending the ideas of \cite{Breu} from the case of the multiplicative group $\mathbb{G}_{m}$ to arbitrary one-dimensional Lubin-Tate groups. For the connection to the different formulations of the $\epsilon$-conjecture in \cite{BB}, \cite{FK}, \cite{Breu}, \cite{BlB} and \cite{BF} see \cite{Iz}. \end{abstract}

\maketitle

\section{Introduction}

Let $\chi:(\mathbb{Z}/N\mathbb{Z})^{\times}\rightarrow \mathbb{C}^{\times}$ be a a group homomorphism called a (primitive) Dirichlet character modulo $N\in\mathbb{N}$ (if $N$ minimal). We extend $\chi$ to $\mathbb{N}$ by setting $\chi(n):=\chi(n \mod N)$ for $(n,N)=1$ and $0$ else. The Dirichlet $L$-function of $\chi$ is defined by
\begin{equation*}
L(\chi,s)=\sum^{\infty}_{n:=1}\frac{\chi(n)}{n^{s}}, \quad s\in \mathbb{C},\; Re(s)>1.
\end{equation*}
The completed $L$-function
\begin{equation*}
\Lambda(\chi,s):=L_{\infty}(\chi,s)L(\chi,s),\quad L_{\infty}(\chi,s)=\Big(\frac{N}{\pi}\Big)^{s/2}\Gamma(\chi,s)
\end{equation*}
admits a meromorphic continuation to $\mathbb{C}$ and satisfies the functional equation
\begin{equation*}
\Lambda(\chi,s)=\frac{\tau(\chi)}{i^{k}\sqrt{N}}\Lambda(\overline{\chi},1-s),\;k\in\left\{0,1\right\},\quad \tau(\chi)=\sum^{N-1}_{\nu:=0}\chi(j)e^{2\pi i\nu/N}.
\end{equation*}
For the trivial character $\chi$ we get the Riemann $\zeta$-function, whose Euler product determines the prime factor decomposition of natural numbers. Using the isomorphism $G(\mathbb{Q}(\mu_{N})/\mathbb{Q})\cong (\mathbb{Z}/N\mathbb{Z})^{\times}$ we can view $\chi$ as a character of $G(\mathbb{Q}(\mu_{N})/\mathbb{Q})$. For $N=4$ we then get the analytic class number formula
\begin{equation*}
h_{\mathbb{Q}(i)}=\frac{\sharp\mu(\mathbb{Q}(i))\sqrt{N}}{2\pi}L(\chi_{1},1), \quad \chi_{1}(\bar{1})=1,\;\chi_{1}(\bar{3})=-1.
\end{equation*}

To arbitrary representations coming from motives $M$ the formulas of this kind have been extended by Bloch and Kato and is called the Tamagawa Number Conjecture (TNC). Again we have the complex $L$-function $L(M, s)$ attached to a motive $M$ which is believed to satisfy the following functional equation relating $L(M, s)$ to the $L$-function $L(M^{*}(1), s)$ of the (Kummer) dual motive $M^{*}(1)$ of $M$:
\begin{equation*}
L(M, s) = \epsilon(M, s)\frac{L_{\infty}(M^{*}(1),-s)}{L_{\infty}(M, s)} L(M^{*}(1),-s).
\end{equation*}

Here $L_{\infty}$ is so the called Euler-factor at infinity (attached to $M$ and $M^{*}(1)$, respectively), which is built up by certain $\Gamma$-factors and certain powers of $2$ and $\pi$ (depending on the Hodge structure of $M$), while the so called $\epsilon$-factor decomposes into local factors
\begin{equation*}
\epsilon(M, s) =\prod_{v}\epsilon_{v}(M, s).
\end{equation*}

Then TNC states a relation between the leading term $L^{*}(0)$ (of the Taylor expansion of $L(M, s)$ at zero s = 0) and certain global Galois cohomology groups $R\Gamma(G_{K},M_{p})$ of $M$ (up to the period and regulator). In the following we assume the validity of the functional equation.
Then it is by no means evident that the validity of the TNC for $M$ is equivalent to the validity of the TNC for $M^{*}(1)$ under this functional
equation on the complex analytic side and under Artin/Verdier or Poitou/Tate-duality on the $p$-adic Galois cohomology side. To the contrary, they are only compatible if and only if $\epsilon_{v}(M, 0)$ are in a specific way related to certain local Galois cohomology groups $R\Gamma(G_{K_{v}},M_{p})$ for all places $v$. This is the content of the absolute $\epsilon$-conjecture.

The equivariant $\epsilon$-conjecture is formulated in the similar way by tensoring the coefficients ($\mathbb{Z}_{p}$-modules) with a group algebra $\mathbb{Z}_{p}[G]${, where $G$ is the Galois group }of some local Galois extension $L_{v}/K_{v}$, i.e., by deforming the presentation along an extension of local fields. This version not only states the relation mentioned above for all twists of $M_{p}$ by Artin representations of the Galois group $G$, simultaneously, but it states a finer congruence measuring the difference between the integral group algebra $\mathbb{Z}_{p}[G]$ and a maximal order in $\mathbb{Q}_{p}[G]$. As before it now describes the compatibility of the Equivariant Tamagawa Number Conjecture (ETNC) \cite{BF} for a motive and its Kummer dual with respect to the functional equation and duality in Galois cohomology. S. Yasuda proved in \cite{Y} the cases in which the residue characteristics of $v$ and $p$ differ, thus we consider only $p$-representations of the absolute Galois groups of local fields.

Recently, Benois and Berger \cite{BB} have proved the equivariant conjecture  with respect to the extension $L/K$ for
arbitrary crystalline representations $V$ of $G_K$, where $K$ is an unramified extension of $\mathbb{Q}_p$
and $L$ a finite subextension of $ K(\mu(p))$ over $K.$  In the special case $V=\mathbb{Q}_p(r),$ $r\in\mathbb{Z}$,
Burns and Flach \cite{BF4} prove a local ETNC using global ingredients in a semi-local setting. Extending work \cite{BlB} of Bley and Burns  Breuning \cite{Breu} proves the equivariant conjecture for $V=\mathbb{Q}_p(1)$ with respect to all tamely ramified extensions.

In this paper we prove the equivariant $\epsilon$-conjecture for tamely ramified extensions   and $p$-adic Tate modules of one-dimensional Lubin-Tate groups defined over $\mathbb{Z}_{p}$ following closely their techniques. For (possibly wildly ramified, $p$-adic Lie) subextensions    of the the maximal abelian extension of $\mathbb{Q}_p$ the corresponding result  has been proved in \cite{ven}.

\begin{flushleft}
{\sc Notation and conventions}
\end{flushleft}
Let $p$ be a prime number. Let $K$ be a finite extension of $\mathbb{Q}_{p}$, $\chi^{ur}:G_{K}\rightarrow \mathbb{Z}_{p}^{\times}$ be a continuous unramified character and $\chi^{cyc}:G_{K}\rightarrow \mathbb{Z}_{p}^{\times}$ denote the $p$-adic cyclotomic character. We consider a continuous representation $T=\mathbb{Z}_{p}(\chi^{ur})(\chi^{cyc})=:\mathbb{Z}_{p}(\chi^{ur})(1)$ of $G_{K}$, which appears naturally as a restriction to $G_{K}$ of the $p$-adic Tate module $\mathrm{T}_{p}\mathcal{F}$ of a one-dimensional Lubin-Tate group $\mathcal{F}$ defined over $\mathbb{Z}_{p}$ (see \cite[Exm.\ 5.20]{Iz}). We set  $V:=\mathbb{Q}_{p}\otimes_{\mathbb{Z}_{p}} T$. Let $L$ be a finite Galois extension of $K$ and let $G=Gal(L/K)$ {  with inertia group $I$; by $Fr_K$ we denote the arithmetic Frobenius homomorphism of $K$}. Let $\widehat{\mathbb{Q}^{ur}_{p}}$ denote the $p$-adic completion of the maximal unramified extension of $\mathbb{Q}_{p}$ and $\widehat{\mathbb{Z}^{ur}_{p}}$ its ring of integers. We set $\Lambda:=\mathbb{Z}_{p}[G]$ and $\Omega:=\mathbb{Q}_{p}[G]$, $\tilde\Lambda:=\widehat{\mathbb{Z}^{ur}_{p}}[G]$ and $\tilde\Omega:=\widehat{\mathbb{Q}^{ur}_{p}}[G]$.

We use the theory of Fontaine's period rings $B_{dR}$, $B_{cris}$, the corresponding functors $\mathrm{D}_{dR}$, $\mathrm{D}_{cris}$ and the subgroups $\mathrm{H}^{1}_{f}$, $\mathrm{H}^{1}_{g}$ etc. The details can be found in \cite{FO}.

We omit the case, in which $\chi^{ur}$ factors over $L$, as in this case the $\epsilon$-conjecture has been proved in \cite{Breu} and assume from now on that $\chi^{ur}(G_{L})\neq 1$.

\section{Galois cohomology}

First we compute the continuous Galois cohomology groups $\mathrm{H}^{i}(L,T)$ as $\Lambda$-modules. We point out, that by this we also determine the Galois cohomology groups of the Kummer dual representation $T^{*}(1)$ in view of the local duality theorems in the Galois cohomology theory.

\begin{proposition} \label{fg:pr:1}
With the notation as above we have:
\begin{enumerate}
\item $\mathrm{H}^{i}(L,T)=0$ for $i\neq 1{,2}$.
\item $\mathrm{H}^{1}(L,T) \cong \Big(\widehat{(L^{ur})^{\times}}^{p}(\chi^{ur})\Big)^{G(L^{ur}/L)}$, where $\widehat{-}^{p}$ denotes the $p$-completion of a group.
\item  There is an isomorphism of $\Lambda$-modules
    \begin{equation*}
\mathcal{F}(\mathfrak{p}_{L}) \overset{\cong}{\longrightarrow} \mathrm{H}^{1}(L,T) .
\end{equation*}
\item {\color{black} $\mathrm{H}^{2}(L,T) \cong \mathrm{H}^0(L,V^*/T^*(1))^\vee\cong \mathbb{Z}_p/p^\omega(\chi^{ur})$ is finite, where $\omega=v_p(1-\chi^{ur}(Fr_L))\neq \infty$. There is an exact sequence of $\Lambda$-modules
    \[\xymatrix{
      0 \ar[r] & {\mathbb{Z}_p[G/I]} \ar[rr]^{1-\chi^{ur}(Fr_K)^{-1}Fr_K} && {\mathbb{Z}_p[G/I]} \ar[rr]^{ } && {\mathrm{H}^{2}(L,T)} \ar[r] & 0 }.\] In particular, $\mathrm{H}^{i}(L,T)$ lies in the category of perfect complexes of $\Lambda$-modules for all $i$.}
\end{enumerate}
\end{proposition}

\begin{proof}
$\mathrm{H}^{i}(L,T)=0$ for $i\neq 0,1,2$ because the cohomological dimension of $G_{K}$ is $2$. Further, $\mathrm{H}^{0}(L,T)=T^{G_{L}}=0$, as the character $\chi^{ur}\otimes \chi^{cyc}:G_{L}\rightarrow \mathbb{Z}_{p}^{\times}$ is not trivial. Using the local duality theorem \cite[Thm.\ (7.2.6)]{NSW} we get {\color{black}$\mathrm{H}^{2}(L,T)\cong\mathrm{H}^0(L,V^*(1)/T^*(1))^\vee$ . Since $V^*(1)/T^*(1)=\mathbb{Q}_{p}/\mathbb{Z}_{p}\big((\chi^{ur})^{-1}\big)$ is unramified, we obtain an
 exact sequence
\small{
\begin{equation*}
\xymatrix{
0 \ar[r] & \mathrm{H}^{0}(L,V^{*}(1)/T^{*}(1)) \ar[r] & \mathbb{Q}_{p}/\mathbb{Z}_{p}\big((\chi^{ur})^{-1}\big) \ar[r]^{\;\;\;1-Fr_{L}} &  \mathbb{Q}_{p}/\mathbb{Z}_{p}\big((\chi^{ur})^{-1}\big) \ar[r] & 0
}\end{equation*}}\normalsize
whose Pontryagin dual becomes
\begin{equation*}
\xymatrix{
0 \ar[r] & \mathbb{Z}_{p}(\chi^{ur})  \ar[r]^{1-Fr_{L}} & \mathbb{Z}_{p}(\chi^{ur}) \ar[r] &  \mathrm{H}^{0}(L,V^{*}(1)/T^{*}(1))^{\vee} \ar[r] & 0,
}\end{equation*}
whence $\mathrm{H}^{2}(L,T)$ is a finite cyclic  group as by assumption $(\chi^{ur})^{-1}(G_{L})\neq 1$.
On the other hand it is easy to calculate the $\mathbb{Z}_p$-elementary divisors of the matrix representing the operator $ 1-\chi^{ur}(Fr_K)^{-1}Fr_K$ on $\mathbb{Z}_p[G/I]$:
\[  1, \ldots , 1, 1-\chi^{ur}(Fr_{L}),\] whence the exactness in (4) is clear. Since $R\Gamma(G_{L},T)$ is a perfect complex, the remaining assertions follow immediately by standard homological algebra.
}

To compute the group $\mathrm{H}^{1}(L,T)$ we use the Hochschild-Serre spectral sequence for the closed subgroup $G_{L^{ur}}$ of $G_{L}$, which exists first only for finite discrete modules $T/p^{n}$, but with \cite[Thm.\ 2.7.5]{NSW} also for the compact module $T$. Note that the character $\chi^{ur}$ factors over $L^{ur}$, such that $G_{L^{ur}}$ acts via the cyclotomic character on $T$. The five-term exact sequence takes the form
\begin{equation*}
\begin{array}{l}
0 \;\rightarrow\; \mathrm{H}^{1}(Gal(L^{ur}/L),T^{G_{L^{ur}}}) \;\rightarrow\; \mathrm{H}^{1}(L,T) \;\rightarrow \\
\\
\mathrm{H}^{1}(G_{L^{ur}},T)^{G(L^{ur}/L)} \;\rightarrow\; \mathrm{H}^{2}(Gal(L^{ur}/L),T^{G_{L^{ur}}}) \;\rightarrow\; \mathrm{H}^{2}(L,T).
\end{array}
\end{equation*}
The module of invariants $T^{G_{L^{ur}}}$ is a zero-module, since the cyclotomic character is not trivial, thus the first and the fourth terms in the exact sequence above vanish and we get a canonical isomorphism
\begin{equation*}
\mathrm{H}^{1}(L,T) \cong \mathrm{H}^{1}(G_{L^{ur}},T)^{G(L^{ur}/L)}=\mathrm{H}^{1}(G_{L^{ur}},\mathbb{Z}_{p}(\chi^{ur})(1))^{G(L^{ur}/L)}.
\end{equation*}

From the Kummer theory and the isomorphism
\begin{equation*}
\mathrm{H}^{1}(G_{L^{ur}},\mathbb{Z}_{p}(\chi^{ur})(1))^{G(L^{ur}/L)}=\Big(\mathrm{H}^{1}(G_{L^{ur}},\mathbb{Z}_{p}(1))(\chi^{ur})\Big)^{G(L^{ur}/L)}
\end{equation*}
we obtain $H^{1}(L,T) \cong \Big(\widehat{(L^{ur})^{\times}}^{p}(\chi^{ur})\Big)^{G(L^{ur}/L)}$.

By taking $G_{L}$-invariants of the exact sequence
\begin{equation} \label{fg:eq:3}
\xymatrix{
0 \ar[r] & \mathcal{F}[p^{n}](\bar{\mathfrak{p}}) \ar[r] & \mathcal{F}(\bar{\mathfrak{p}}) \ar[r]^{[p^{n}]}  & \mathcal{F}(\bar{\mathfrak{p}}) \ar[r] & 0 \\
} \end{equation}
we get the following exact sequence of $\Lambda$-modules
\begin{equation*}
\xymatrix{
0 \ar[r] & ^{\mathcal{F}(\mathfrak{p}_{L})}/_{[p^{n}](\mathcal{F}(\mathfrak{p}_{L}))}  \ar[r]  & \mathrm{H}^{1}(L, \mathcal{F}[p^{n}](\bar{\mathfrak{p}})) \ar[r] &  \mathrm{H}^{1}(L, \mathcal{F}(\bar{\mathfrak{p}}))[p^{n}] \ar[r] & 0
} \end{equation*}
for each $n\geq 1$.

The inverse limit over $n$ of the exact sequences above results in the exact sequence of $\Lambda$-modules
\begin{equation*}
\xymatrix{
0 \ar[r] & \mathcal{F}(\mathfrak{p}_{L}) \ar[r]  & \mathrm{H}^{1}(L, T) \ar[r] &  \mathrm{H}^{1}(L, T)/\mathcal{F}(\mathfrak{p}_{L}) \ar[r] & 0
} \end{equation*}
$\mathcal{F}(\mathfrak{p}_{L})$ being a f.g.\ $\mathbb{Z}_{p}$-module (cf.\ \cite[4.5.1]{Breu}). From \cite[Exm.\ 5.20]{Iz} we deduce that the quotient $\mathrm{H}^{1}(L, T)/\mathcal{F}(\mathfrak{p}_{L})$ is isomorphic to
\begin{equation} \label{awful quotient}
 \Big(\frac{\widehat{(L^{ur})^{\times}}^{p}}{U^{1}(\widehat{L^{ur}})}(\chi^{ur})\Big)^{G(L^{ur}/L)}.
\end{equation}

For the extension $L^{ur}/\mathbb{Q}_{p}$ we have (both algebraically and topologically)
\begin{equation*}
(L^{ur})^{\times}=(\pi_{L})\times \mathcal{O}^{\times}_{L^{ur}}\cong \mathbb{Z}\oplus \mathcal{O}^{\times}_{L^{ur}},
\end{equation*}
where $\pi_{L}$ is a prime element of $\mathcal{O}_{L}$. Let $\kappa_{L}$ denote the residue class field of $L$, then we have a split exact sequence
\begin{equation*}
\xymatrix{
1 \ar[r] & U^{1}(L^{ur}) \ar[r] & \mathcal{O}^{\times}_{L^{ur}} \ar[r] & \overline{\kappa_{L}}^{\times} \ar[r] & 1.
}\end{equation*}
The group $\overline{\kappa_{L}}^{\times}$ is $p$-divisible, thus
\begin{equation*}
\widehat{\overline{\kappa_{L}}^{\times}}^{p}:=\underset{n}{\varprojlim}\;\overline{\kappa_{L}}^{\times}/(\overline{\kappa_{L}}^{\times})^{p^{n}}=1
\end{equation*}
and
\begin{equation*}
 \widehat{(L^{ur})^{\times}}^{p}=\widehat{(\pi_{L})}^{p}\times \widehat{U^{1}(L^{ur})}^{p}=\widehat{(\pi_{L})}^{p}\times U^{1}(\widehat{L^{ur}}),
\end{equation*}
whence the quotient in \eqref{awful quotient} is isomorphic to
\begin{equation*}
\Big(\widehat{(\pi_{L})}^{p}(\chi^{ur})\Big)^{G(L^{ur}/L)}.
\end{equation*}
The last is zero, since the group $G(L^{ur}/L)$ acts trivially on $(\pi_{L})$ and $\chi^{ur}$ is a non-trivial character .
\end{proof}

\begin{flushleft}
\quad Next we compute the finite part $\mathrm{H}^{1}_{f}(L,T)\subseteq \mathrm{H}^{1}(L,T)$ defined as a preimage of $\mathrm{H}^{1}_{f}(L,V)$  under the map $i:\mathrm{H}^{1}(L,T)\rightarrow \mathrm{H}^{1}(L,V)$.
\end{flushleft}

\begin{lemma} \label{fg:le:1}
$\dim_{\mathbb{Q}_{p}}\mathrm{H}^{1}_{f}(L,V)=\dim_{\mathbb{Q}_{p}}\mathrm{H}^{1}(L,V)=[L:\mathbb{Q}_{p}]$.
\end{lemma}

\begin{proof}
Both, $V$ and $V^{*}(1)$, are de Rham representations of $G_{L}$, thus from \cite[pp. 355-356]{BK} we have
\begin{equation} \label{fg:eq:1}
\dim_{\mathbb{Q}_{p}} \mathrm{H}^{1}_{f}(L,V)+\dim_{\mathbb{Q}_{p}} \mathrm{H}^{1}_{f}(L,V^{*}(1))=\dim_{\mathbb{Q}_{p}} \mathrm{H}^{1}(L,V);
\end{equation}
\begin{equation} \label{fg:eq:2}
\dim_{\mathbb{Q}_{p}} \mathrm{H}^{1}_{f}(L,V)=\dim_{\mathbb{Q}_{p}} (t_{V}(L))+\dim_{\mathbb{Q}_{p}} \mathrm{H}^{0}(L,V),
\end{equation}
where $t_{V}(L):=\mathrm{D}^{L}_{dR}(V)/Fil^{0}\mathrm{D}^{L}_{dR}(V)$. The same is true for $V^{*}(1)$.

But $\mathrm{H}^{0}(L,V)$ and $\mathrm{H}^{0}(L,V^{*}(1))$ are zeros by the proof of Proposition \ref{fg:pr:1}, so that
\begin{equation*}
\dim_{\mathbb{Q}_{p}} \mathrm{H}^{1}_{f}(L,V)=\dim_{\mathbb{Q}_{p}} (t_{V}(L))
\end{equation*}
and
\begin{equation*}
\dim_{\mathbb{Q}_{p}} \mathrm{H}^{1}_{f}(L,V^{*}(1))=\dim_{\mathbb{Q}_{p}} (t_{V^{*}(1)}(L)).
\end{equation*}
For a de Rham representation $W$ by \cite[p. 148]{FO}
\begin{equation*}
t_{W}(L)=gr^{-1}(W)\hookrightarrow (\mathbb{C}_{p}(-1)\otimes_{\mathbb{Q}_{p}}W)^{G_{L}}.
\end{equation*}
Moreover, by Corollary 3.57 in (loc.\ cit.)
\begin{equation*}
(\mathbb{C}_{p}(-1)\otimes_{\mathbb{Q}_{p}}V)^{G_{L}}=(\mathbb{C}_{p}(\chi^{ur}))^{G_{L}}\cong L
\end{equation*}
and
\begin{equation*}
(\mathbb{C}_{p}(-1)\otimes_{\mathbb{Q}_{p}}V^{*}(1))^{G_{L}}=(\mathbb{C}_{p}(\chi^{ur})^{-1}(-1))^{G_{L}}=0,
\end{equation*}
thus from the equality \eqref{fg:eq:1} we get
\begin{equation*}
\dim_{\mathbb{Q}_{p}} \mathrm{H}^{1}_{f}(L,V)=\dim_{\mathbb{Q}_{p}} \mathrm{H}^{1}(L,V).
\end{equation*}
Finally, using the formula for the Euler characteristic
\begin{equation*}
\sum_{i=0}^{\infty}  (-1)^{i} \dim_{\mathbb{Q}_{p}} \mathrm{H}^{i}(L,V)=-[L:\mathbb{Q}_{p}]\cdot \dim_{\mathbb{Q}_{p}} V
\end{equation*}
we see that $\dim_{\mathbb{Q}_{p}} \mathrm{H}^{1}(L,V)=[L:\mathbb{Q}_{p}]$, as $\mathrm{H}^{i}(L,V)=0$ for $i\neq 1$ (cf. Proposition \ref{fg:pr:1}).
\end{proof}

\begin{corollary} \label{cor:120}
From the above proposition it follows that
\begin{enumerate}
\item $\mathrm{H}^{1}_{f}(L,T)=\mathrm{H}^{1}(L,T)$ is a $\mathbb{Z}_{p}$-module of rank $[L:\mathbb{Q}_{p}]$,
\item $\mathrm{H}^{1}_{f}(L,T^{*}(1))=\mathrm{H}^{1}(L,T^{*}(1))_{tors}\cong \mathrm{H}^{0}(L,V^{*}(1)/T^{*}(1))$ is a finite torsion group.
\end{enumerate}
\end{corollary}

\begin{proof}
The first part is obvious. By the definition of $\mathrm{H}^{1}_{f}(L,T^{*}(1))$ it contains the torsion subgroup of $\mathrm{H}^{1}(L,T^{*}(1))$ and, since the image of $\mathrm{H}^{1}_{f}(L,T^{*}(1))$ in $\mathrm{H}^{1}(L,V^{*}(1))$ is zero, they are equal. Consider an exact sequence of $G_{L}$-modules
\begin{equation*}
\xymatrix{
0 \ar[r] & T^{*}(1) \ar[r] & V^{*}(1) \ar[r] & V^{*}(1)/T^{*}(1) \ar[r] & 0.
}\end{equation*}
The associated long exact sequence in cohomology is
\begin{equation*}
\xymatrix{
0 \ar[r] & \mathrm{H}^{0}(L,T^{*}(1)) \ar[r] & \mathrm{H}^{0}(L,V^{*}(1)) \ar[r] & \mathrm{H}^{0}(L,V^{*}(1)/T^{*}(1)) \ar[r] & \\
 \ar[r] & \mathrm{H}^{1}(L,T^{*}(1)) \ar[r] & \mathrm{H}^{1}(L,V^{*}(1)) \ar[r] & \mathrm{H}^{1}(L,V^{*}(1)/T^{*}(1)) \ar[r] & \ldots\\
}\end{equation*}

The groups $\mathrm{H}^{0}(L,T^{*}(1))$ and $\mathrm{H}^{0}(L,V^{*}(1))$ are zeros. Further, the third group $\mathrm{H}^{0}(L,V^{*}(1)/T^{*}(1))$ is a finite torsion group, since $(\chi^{ur})^{-1}\equiv \mathrm{id}\;(\mathrm{mod}\; p^{k})$ for some $k>>1$, so that we can replace $\mathrm{H}^{1}(L,T^{*}(1))$ by $\mathrm{H}^{1}_{f}(L,T^{*}(1))$ in the exact sequence above getting
\begin{equation*}
\xymatrix{
0 \ar[r] & \mathrm{H}^{0}(L,V^{*}(1)/T^{*}(1)) \ar[r]^{\cong} & \mathrm{H}^{1}_{f}(L,T^{*}(1)) \ar[r] & 0.
}\end{equation*}
\end{proof}

\section{Comparison isomorphism}
Let $T$ be a finitely generated (free) $\zp$-module with a continuous $G_K$-action. We assume that $V:=\qp\otimes_{\zp} T$ is a de Rham representation. Furthermore we define $\zp[G]$-modules \[\T:=\zp[G]^\sharp \otimes_{\zp} T\] and \[\V:=\zp[G]^\sharp \otimes_{\zp} V\] where the action is given by left-multiplication on the left tensor-factor. The sharp indicates that these modules are endowed with the following $G_K$-actions: $\sigma(\lambda\otimes t):= \lambda \bar{\sigma}^{-1}\otimes \sigma t,$ where $\bar{\sigma}$ denotes the image of $\sigma$ under the natural projection map. Henceforth  we use the following explicit realisation for the induction
\[ \mathrm{Ind}_{L/\mathbb{Q}_{p}}T:=\zp[G_{\qp}]\otimes_{\zp[G_K]} \T \left(\cong \zp[G_{\qp}]\otimes_{\zp[G_L]} T\right)\]
and similarly for $\mathrm{Ind}_{L/\mathbb{Q}_{p}}V$. These are $G_{\qp}$-modules by the action on the left tensor-factor while they become $\Lambda$- and $\Omega$-modules via the corresponding module structures of $\T$ and $\V$, respectively.

The $p$-adic comparison isomorphism for   $V$
\begin{equation*}
\begin{array}{rcl}
comp_{V,L/\mathbb{Q}_{p}}:B_{dR}\otimes_{\mathbb{Q}_{p}}\mathrm{D}_{dR}(\mathrm{Ind}_{L/\mathbb{Q}_{p}}V)& \overset{\cong}{\rightarrow} & B_{dR}\otimes_{\mathbb{Q}_{p}} \mathrm{Ind}_{L/\mathbb{Q}_{p}}V,\\
\\
c\otimes x & \mapsto & c x
\end{array}
\end{equation*}
is a $B_{dR}[G]$-linear map, which commutes with the action of $G_{\mathbb{Q}_{p}}$, where $G_{\mathbb{Q}_{p}}$ acts on $B_{dR}\otimes_{\mathbb{Q}_{p}}\mathrm{D}_{dR}(\mathrm{Ind}_{L/\mathbb{Q}_{p}}V)$ via $g(c\otimes x)=g(c)\otimes x$ and diagonally on $B_{dR}\otimes_{\mathbb{Q}_{p}} \mathrm{Ind}_{L/\mathbb{Q}_{p}}V$.

Let $\xi=(\xi_n)$ be a compatible system of $p^{n}$th roots of unity and $t:=log[\xi]\in B_{dR}$ denote the $p$-adic period analogous to $2\pi i$, then $g(t)=\chi^{cyc}(g)\cdot t$ for all $g\in G_{\mathbb{Q}_{p}}$. We apply (not necessary commutative) determinant functor to $comp_{V,L/\mathbb{Q}_{p}}$ to obtain a map
\begin{equation*}
\tilde{\alpha}_{V,L/K}=(x,y)\in \mathrm{Isom}\left(\mathrm{\mathbf{d}}_{\Omega}(\mathrm{D}_{dR}(\mathrm{Ind}_{L/\mathbb{Q}_{p}}V)),
\mathrm{\mathbf{d}}_{\Omega}(\mathrm{Ind}_{L/\mathbb{Q}_{p}}V)\right)\times^{K_{1}(\Omega)}K_{1}(B_{dR}[G]).
\end{equation*}
Multiplying $\tilde{\alpha}_{V,L/K}$ with $t $ we get
\small{
\begin{equation*}
\alpha_{V,L/K}=(x,ty) \in \mathrm{Isom}\left(\mathrm{\mathbf{d}}_{\Omega}(\mathrm{D}_{dR}(\mathrm{Ind}_{L/\mathbb{Q}_{p}}V)),
\mathrm{\mathbf{d}}_{\Omega}(\mathrm{Ind}_{L/\mathbb{Q}_{p}}V)\right)\times^{K_{1}(\Omega)}K_{1}(\widehat{L^{ur}}[G])
\end{equation*}} \normalsize
with
\begin{equation*}
g(y)=[\mathrm{Ind}_{L/\mathbb{Q}_{p}} V,g]\cdot y,\;\forall g\in G(L^{ur}/\mathbb{Q}_{p}).
\end{equation*}

The maximal abelian extension $\mathbb{Q}^{ab}_{p}$ of $\mathbb{Q}_{p}$ is the composite of the maximal unramified extension $\mathbb{Q}^{ur}_{p}$ and the cyclotomic extension $\mathbb{Q}_{p,\infty}$, which is obtained by adjoining all $p$-power roots of $1$. For $g\in G_{\mathbb{Q}_{p}}$ we define $g^{ur}\in Gal(\mathbb{Q}^{ab}_{p}/\mathbb{Q}_{p})$ by $g^{ur}|_{\mathbb{Q}^{ur}_{p}}=g|_{\mathbb{Q}^{ur}_{p}}$ and $g^{ur}|_{\mathbb{Q}_{p,\infty}}=\mathrm{id}$. We also define $g^{ram}\in Gal(\mathbb{Q}^{ab}_{p}/\mathbb{Q}_{p})$ by $g^{ram}|_{\mathbb{Q}^{ur}_{p}}=\mathrm{id}$ and $g^{ram}|_{\mathbb{Q}_{p,\infty}}=g|_{\mathbb{Q}_{p,\infty}}$. Thus $g|_{\mathbb{Q}^{ab}_{p}}=g^{ur}g^{ram}$.

Let $\Gamma(V):=\prod_\mathbb{Z} \Gamma^*(j)^{-h(-j)},$ $\Theta_p(V)$   where $h(j)=\dim_{\mathbb{Q}_p} gr^j D_{dR}(V)$ and  where $\Gamma^*(-j)$ is defined to be
$\Gamma(j)=(j-1)!$ if $j>0$ and $\lim_{s\to j}
(s-j)\Gamma(s)=(-1)^{j}((-j)!)^{-1}$ otherwise.

We set
\begin{equation*}
\beta_{V,L/K}:=\Gamma(V)\epsilon_{D}(L/K,V) \cdot\alpha_{V,L/K}=(x,\Gamma(V)\epsilon_{D}(L/K,V) \cdot t  \cdot y)=:(x,\tilde{y}),
\end{equation*}
where
\begin{equation*}
\epsilon_{D}(L/K,V):=\Big(\epsilon(\mathrm{D}_{pst}(\mathrm{Ind}_{K/\mathbb{Q}_{p}}(V\otimes \rho^{*}_{\chi})),\psi_\xi,dx)\Big)_{\chi\in Irr(G)}\in K_{1}(\overline{\mathbb{Q}_{p}}[G])
\end{equation*}
where  $\rho^{}_{\chi}$ denotes a representation with character $\chi$ and $\rho^{*}_{\chi}$ its dual.
The right hand side is defined similar to \cite[pp.\ 21-22]{BB} or \cite[3.3.3]{FK} via the theory of local $\epsilon$-constants \`a  la Deligne \cite{del-localconstant}, see also Tate \cite{Ta1}. In particular, our convention is that under the local reciprocity law uniformisers $\pi_K$ correspond to   {\it geometric} Frobenius automorphisms.  Assume that all the   characters of $G$ are defined over  the fixed finite extension $F$ of $\mathbb{Q}_p$. Then we may either fix an embedding $\iota: F\hookrightarrow \overline{\mathbb{Q}_p}$  and use the Weil-Deligne representation $\overline{\mathbb{Q}_p}\otimes_{\mathbb{Q}^{ur}\otimes_{\mathbb{Q}_p}F}\mathrm{D}_{pst}(\mathrm{Ind}_{K/\mathbb{Q}_{p}}(V\otimes \rho^{*}_{\chi}))$ - as we do here - or one can work with all such embeddings simultaneously as in \cite[3.3.3]{FK}. Anyway we choose $\psi_\xi$ to be the additive character associated with the above $\xi$, i.e., $\psi_\xi(p^{-n})=\xi_n,$ $n\geq 1,$ and conductor\footnote{$n(\psi_\xi)$ is defined to be the largest integer $n$ such that $\psi_\xi(p^{-n}\mathbb{Z}_{p})=1$.} $n(\psi_\xi)=0$ as well as $dx$ to the Haar measure on $\mathbb{Q}_p$ such that $\mathbb{Z}_p$ has measure $1$. Hence we often will omit this data in the notation.

Note also that we have canonical isomorphisms
\[V_\rho\otimes_\Lambda \Ind_{L/\qp} T\cong  \zp[G_{\qp}]\otimes_{\zp[G_K]}  (V_\rho\otimes_\Lambda \T)\cong  \zp[G_{\qp}]\otimes_{\zp[G_K]}(V_{\rho^*}\otimes_{\qp} V),\]
whence the $\rho$-component above in the definition of the $\epsilon$-factor is the one which is used by Fukaya and Kato \cite{FK} for the module $V_\rho\otimes_\Lambda \Ind_{L/\qp} T$, compare also with appendix \ref{app}.

According to \cite[Lem.\ 2.4.3]{BB} $\beta_{V,L/K}$ is an element of
\begin{equation*}
\mathrm{Isom}\left(\mathrm{\mathbf{d}}_{\Omega}(\mathrm{D}_{dR}(\mathrm{Ind}_{L/\mathbb{Q}_{p}}V)),
\mathrm{\mathbf{d}}_{\Omega}(\mathrm{Ind}_{L/\mathbb{Q}_{p}}V)\right)\times^{K_{1}(\Omega)}K_{1}(\widehat{L^{ur}}[G])
\end{equation*}
with
\begin{equation*}
g(\tilde{y})=g^{ur}(\tilde{y})=[\mathrm{Ind}_{L/\mathbb{Q}_{p}} V,g^{ur}]\cdot \tilde{y},\;\forall g\in G(L^{ur}/\mathbb{Q}_{p}),
\end{equation*}
i.e. $\tilde{y}\in K_{1}(\tilde{\Omega})_{  [\mathrm{Ind}_{L/\mathbb{Q}_{p}} V,g^{ur}]   }$ in the terminology of \cite{FK}.

{  For a representation $W$ of $G_K$ we denote by $\frak{f}(W)=\pi_K^{a(W)}$ its  local Artin conductor while $f(W)=q_K^{a(W)}=|\frak{f}(W)|_p^{-1}$ denotes its absolute norm, where $q_K$ is the cardinality of the residue class field of $K.$ Setting $\pi_{\mathbb{Q}_p}=p$ we have  $f(W)=p^{a(W)}=\frak{f}(W)$ for every $G_{\mathbb{Q}_p}$-representations $W$.}

{  
For $a\in F^\times$ and an additive character $\psi$ of $F$ we denote by $a\psi$ the character sending $x$ to $\psi(ax).$
\begin{lemma}\label{lemma} We have the equality
\begin{eqnarray*}
  \epsilon(\mathrm{D}_{pst}(\mathrm{Ind}_{K/\mathbb{Q}_{p}}(V\otimes \rho_{\chi}^*)),\psi_\xi)&=&\epsilon( \mathrm{Ind}_{K/\mathbb{Q}_{p}}(  \rho^{ }_{\chi}), -\psi_\xi)^{-1} \chi^{ur}(\frak{f}(\mathrm{Ind}_{K/\mathbb{Q}_{p}}(  \rho^{ }_{\chi})))\\
 \end{eqnarray*}
where   $\mathbb{Z}_{p}(\chi^{ur} )=\mathrm{T}_{p}\mathcal{F}(-1)$  and the character $\chi^{ur}=\chi^{ur}_{\mathbb{Q}_{p}}$ is viewed as a character $\mathbb{Q}_{p}^{\times}\rightarrow G^{ab}_{\mathbb{Q}_{p}}\overset{\chi^{ur}_{\mathbb{Q}_{p}}}{\rightarrow} \mathbb{Q}_{p}^{\times}$ via the local reciprocity law sending $p$ to the geometric Frobenius.
\end{lemma}

\begin{proof}
 Since $\mathrm{D}_{pst}(\mathrm{Ind}_{K/\mathbb{Q}_{p}}(V\otimes \rho^{*}_{\chi}))\cong \mathrm{D}_{pst}(V)\otimes_{\mathbb{Q}^{ur}}\mathrm{D}_{pst}(\mathrm{Ind}_{K/\mathbb{Q}_{p}}(  \rho^{*}_{\chi}))$ we obtain for the linearized Kummer dual \[\mathrm{D}_{pst}(\mathrm{Ind}_{K/\mathbb{Q}_{p}}(V\otimes \rho^{*}_{\chi}))^*(1)\cong \mathrm{D}_{pst}(V)^*(1)\otimes_{\mathbb{Q}^{ur}}\mathrm{D}_{pst}(\mathrm{Ind}_{K/\mathbb{Q}_{p}}(  \rho^{*}_{\chi}))^*.\]
 Now  the (linearised) action on the unramified $\mathrm{D}_{pst}(V)$ is given by the character $g\mapsto (p\chi^{ur}(Fr_{\mathbb{Q}_p}))^{v(g)}$, whence $\mathrm{D}_{pst}(V)^*(1)$ bears the action $g\mapsto \chi^{ur}(Fr_{\mathbb{Q}_p})^{-v(g)}$ while the Weil-Deligne representations $\mathrm{D}_{pst}(\mathrm{Ind}_{K/\mathbb{Q}_{p}}(  \rho_{\chi}^*))^{*}$ and $\mathrm{Ind}_{K/\mathbb{Q}_{p}}(  \rho^{ }_{\chi})$ can be identified. Now the claim follows from \cite[3.2.2 (3) and (5)]{FK} where $\tau$ denotes a geometric Frobenius.
\end{proof}

}



The induction property of local Artin conductors (see \cite[Lem.\ 3.3]{Breu}) gives
\begin{equation} \label{eq:118}
 f(\mathrm{Ind}_{K/\mathbb{Q}_{p}}\rho^{*}_{\chi}) = N_{K/\mathbb{Q}_{p}}(f(\rho^{*}_{\chi}))\cdot d_{K/\mathbb{Q}_{p}}^{\chi(1)}
\end{equation}
$d_{K/\mathbb{Q}_{p}}$ being the discriminant of $K/\mathbb{Q}_{p}$.

For later purposes we also introduce the local Galois Gauss sum
\[\tau_{K}(\chi)=\epsilon(\rho_\chi|\cdot|_p^{\frac{1}{2}}, \psi_{K},dx_{ \psi_{K}})\sqrt{ f(\chi)},\] i.e.,   $dx_{\psi_{K}}$ is the Haar measure of $K$ which is selfdual with respect to the standard additive character $\psi_{K}:=\psi_{\mathbb{Q}_p}\circ Tr_{K/\mathbb{Q}_p}$. {\color{red} Note that this definition coincides with that in \cite{Breu,Fr} although in Breuning's thesis $\overline{\rho}$ instead of $\rho$ shows up   in the definition of $\tau$:   there implicitly the norm rest symbol is normalised by sending $\pi_K$ to the arithmetic Frobenius automorphism, i.e., opposite to the convention used here, whence one has to replace a representation by its contragredient.}

\begin{remark} \label{cor. term}
An easy calculation shows that in our case the factor $\Gamma(V)$, i.e.,   $\Gamma_{L}(V)$ of \cite[3.3.4]{FK}  or $\Gamma^{*}(V)$ of \cite[2.4]{BB}, used for the correction of the comparison isomorphism is equal to $1$.
\end{remark}

\section{Bloch-Kato exponential map}

In this section we follow closely the approach of \cite{BB} to construct an isomorphism
\begin{equation*}
\tilde{\epsilon}_{\Omega,\xi}(\mathrm{Ind}_{L/\mathbb{Q}_{p}} V):\mathrm{\mathbf{d}}_{\tilde{\Omega}}(0)\rightarrow \tilde{\Omega}\otimes_{\Omega}\left\{ \mathrm{\mathbf{d}}_{\Omega} (R\Gamma(L,V)) \cdot \mathrm{\mathbf{d}}_{\Omega}(\mathrm{Ind}_{L/\mathbb{Q}_{p}} V)\right\}
\end{equation*}
satisfying the following condition \newline
$(\star):\quad$ Let $\rho:\Omega\rightarrow GL_{n}(F),\;n\geq 1,\;[F:\mathbb{Q}_{p}]<\infty$ be a continuous representation. Then the image of $\tilde{\epsilon}_{\Omega,\xi}(\mathrm{Ind}_{L/\mathbb{Q}_{p}} V)$ under $F^{n}\otimes_{\Omega}-$ is the $\epsilon$-isomorphism of de Rham representations described in \cite[Sec.\ 3.3]{FK}.

Consider the exact sequence of $\mathbb{Q}_{p}[G]$-modules (see \cite[2.5.1]{BB}):
\begin{equation*}
\xymatrix{
0 \ar[r] & \mathrm{H}^{0}(L,V) \ar[r] & \mathrm{D}^{L}_{cris}(V) \ar[r] & \mathrm{D}^{L}_{cris}(V)\oplus t_{V}(L) \ar[r]^{\quad \;exp_{V}} & \mathrm{H}^{1}(L,V) \ar[r] & \;
} \end{equation*}
\begin{equation} \label{eq:111}
\xymatrix{
\;\ar[r]^{exp^{*}_{V^{*}(1)}\quad\quad\quad\quad\quad\;\;} & \mathrm{D}^{L}_{cris}(V^{*}(1))^{*}\oplus t^{*}_{V^{*}(1)}(L) \ar[r] & \mathrm{D}^{L}_{cris}(V^{*}(1))^{*} \ar[r] & \mathrm{H}^{2}(L,V) \ar[r] & 0,
} \end{equation}
where $exp_{W}:t_{W}(L)\rightarrow \mathrm{H}^{1}_{f}(L,W)$ is the Bloch-Kato exponential map for a de Rham representation $W$. By the proof of Lemma \ref{fg:le:1} we know that $\mathrm{H}^{0}(L,V)$, $\mathrm{H}^{2}(L,V)$, $\mathrm{H}^{1}_{f}(L,V^{*}(1))$, $t_{V^{*}(1)}(L)$ are zeros, whence the exact sequence above degenerates to
\begin{equation*}
\xymatrix{ 0\ar[r] & \mathrm{D}^{L}_{cris}(V^{*}(1))^{*} \ar[r]^{1-\phi^{*}} & \mathrm{D}^{L}_{cris}(V^{*}(1))^{*} \ar[r] & 0
}\end{equation*}
and
\begin{equation*}
\xymatrix{
0 \ar[r] & \mathrm{D}^{L}_{cris}(V) \ar[r] & \mathrm{D}^{L}_{cris}(V)\oplus t_{V}(L) \ar[r]^{\;\;\;\;\;\;exp_{V}} & \mathrm{H}^{1}(L,V)   \ar[r] &  0,
} \end{equation*}
where $\phi\in\mathrm{End}(\mathrm{D}^{L}_{cris}(-))$ is induced by an endomorphism of the ring $B_{cris}$.  Further, using the exact sequence
\begin{equation*}
\xymatrix{
0 \ar[r] & t^{*}_{V^{*}(1)}(L) \ar[r] & \mathrm{D}^{L}_{dR}(V) \ar[r] & t_{V}(L) \ar[r] & 0
}\end{equation*}
($t^{*}_{V^{*}(1)}(L)\cong Fil^{0}\mathrm{D}^{L}_{dR}(V)$) and the isomorphism (of \cite[Lem.\ 1.4.1]{BB})
\begin{equation*}
\frac{\mathrm{D}^{L}_{cris}(V)}{(1-\phi)\mathrm{D}^{L}_{cris}(V)}\cong \frac{\mathrm{H}^{1}_{f}(L,V)}{\mathrm{H}^{1}_{e}(L,V)}=0
\end{equation*}
we deduce that
\begin{equation*}
\mathrm{D}^{L}_{dR}(V)\overset{exp}{\longrightarrow} \mathrm{H}^{1}(L,V) \;\text{ and }\; \mathrm{D}^{L}_{cris}(V)\overset{1-\phi}{\longrightarrow} \mathrm{D}^{L}_{cris}(V)
\end{equation*}
are isomorphisms. The application of the determinant functor to $\eqref{eq:111}$ results in the isomorphism $\mathrm{\mathbf{d}}_{\Omega}(1-\phi)\cdot\mathrm{\mathbf{d}}_{\Omega}(exp^{-1})\cdot \mathrm{\mathbf{d}}_{\Omega}((1-\phi^{*})^{-1})$ sending
\begin{equation*}
\mathrm{\mathbf{d}}_{\Omega} (\mathrm{D}^{L}_{cris}(V))\cdot \mathrm{\mathbf{d}}_{\Omega} (R\Gamma(L,V))^{-1} \cdot \mathrm{\mathbf{d}}_{\Omega} (\mathrm{D}^{L}_{cris}(V^{*}(1))^{*})
\end{equation*}
to
\begin{equation*}
 \mathrm{\mathbf{d}}_{\Omega} (\mathrm{D}^{L}_{cris}(V))\cdot\mathrm{\mathbf{d}}_{\Omega}(\mathrm{D}^{L}_{dR}(V)) \cdot   \mathrm{\mathbf{d}}_{\Omega} (\mathrm{D}^{L}_{cris}(V^{*}(1))^{*}),
\end{equation*}
which after the composition with
\begin{equation*}
\mathrm{id}_{\mathrm{\mathbf{d}}_{\Omega} (\mathrm{D}^{L}_{cris}(V))} \cdot \beta_{V,L/K}\cdot \mathrm{id}_{\mathrm{\mathbf{d}}_{\Omega} (\mathrm{D}^{L}_{cris}(V^{*}(1))^{*})}
\end{equation*}
and multiplication with
\begin{equation*}
\mathrm{id}_{\mathrm{\mathbf{d}}_{\Omega} (\mathrm{D}^{L}_{cris}(V))^{-1}}\cdot\mathrm{id}_{\mathrm{\mathbf{d}}_{\tilde{\Omega}} (R\Gamma(L,V))}\cdot\mathrm{id}_{\mathrm{\mathbf{d}}_{\tilde{\Omega}} (\mathrm{D}^{L}_{cris}(V^{*}(1))^{*})^{-1}}
\end{equation*}
gives the desired isomorphism $\tilde{\epsilon}_{\Omega,\xi}(\mathrm{Ind}_{L/\mathbb{Q}_{p}} V)$. Note that $\tilde{\epsilon}_{\Omega,\xi}(\mathrm{Ind}_{L/\mathbb{Q}_{p}} V)$ satisfies the condition $(\star)$ automatically by construction (cf. the construction of the $\epsilon$-isomorphisms of de Rham representations in \cite[Sec.\ 3.3]{FK} and Remark \ref{cor. term}).

Let $\mathcal{G}$ be a commutative formal Lie group of finite height over $\mathcal{O}_{K}$ and $W$ be a $p$-adic de Rham representation coming from the $p$-adic Tate module of $\mathcal{G}$. In \cite[pp.\ 359-360]{BK} is described a commutative diagram, which connects the Bloch-Kato exponential map with the classical exponential map of $\mathcal{G}$:
\begin{equation*}
\xymatrix{
\mathrm{tan}(\mathcal{G}_{K})(L) \ar[r]^{exp} \ar[d]^{=} & \mathcal{G}(\mathfrak{p}_{L})\otimes \mathbb{Q}_{p} \ar[d] \\
t_{W}(L) \ar[r]^{exp} & \mathrm{H}^{1}(L,W),
} \end{equation*}
where $t_{W}(L)$ is identified with the tangent space of $\mathcal{G}_{K}$, the upper (resp. lower) $exp$ is the exponential map in the classical sense (resp. Bloch-Kato) and the right vertical map is the boundary map of the Kummer sequence \eqref{fg:eq:3}.

By Proposition \ref{fg:pr:1}(3) and Lemma \ref{fg:le:1} the representation $T$ being the $p$-adic Tate module of a formal group $\mathcal{F}$ (see \cite[Exm.\ 5.20]{Iz}) fits into the commutative diagram

\begin{equation} \label{log:eq:1}
\xymatrix{
  \mathcal{F}(\mathfrak{p}_{L}) \ar[r]^{\cong} \ar[d] & \mathrm{H}^{1}(L,T)   \ar[d]^{i}    \\
   t_{V}(L) \ar[r]^{exp\;\;}_{\cong} & \mathrm{H}^{1}(L,V)  ,
} \end{equation}
where the left vertical arrow is a $\Lambda$-homomorphism induced by the classical logarithm $log_{\mathcal{F}}$ of $\mathcal{F}$.

\section{Formulation of the $\epsilon$-conjecture}

In this section we follow closely the approach of \cite{Breu} to formulate an $\epsilon$-conjecture in the language of relative $K_{0}$-groups. In the previous section we constructed an isomorphism
\begin{flushleft}
\small{
\begin{equation*}
\tilde{\epsilon}_{\Omega,\xi}(\mathrm{Ind}_{L/\mathbb{Q}_{p}} V):=(x,\tilde{y})\in \mathrm{Isom}(\mathrm{\mathbf{d}}_{\Omega}(0),\mathrm{\mathbf{d}}_{\Omega} (R\Gamma(L,V)) \cdot \mathrm{\mathbf{d}}_{\Omega}(\mathrm{Ind}_{L/\mathbb{Q}_{p}} V))\times^{K_{1}(\Omega)}K_{1}(\tilde\Omega)
\end{equation*}}\normalsize
\end{flushleft}
satisfying condition $(\star)$. {   Here we can and do assume that $x$ arises by base change from an isomorphism $ \mathrm{\mathbf{d}}_{\Lambda}(0)\cong\mathrm{\mathbf{d}}_{\Lambda} (R\Gamma(L,T)) \cdot \mathrm{\mathbf{d}}_{\Lambda}(\mathrm{Ind}_{L/\mathbb{Q}_{p}} T)$.}

In view of \cite[Prop.\ 3.1.3]{FK} and of the localization exact sequence for $K$-groups
\begin{equation*}
\xymatrix{
1 \ar[r] & K_{1}(\tilde\Lambda) \ar[r] & K_{1}(\tilde\Omega) \ar[r]^{\partial\;\;\;} & K_{0}(\tilde\Lambda,\widehat{\mathbb{Q}^{ur}_{p}}) \ar[r] & 0
}\end{equation*}
($SK_{1}(\tilde\Lambda)$ being trivial by \cite[Cor.\ 2.28]{iz-ven} and the map $\partial$ being surjective by \cite[Lem.\ 2.5 and Sec.\ 2.4.4]{Breu}) we formulate
\begin{conjecture}[$C^{na}_{ep}(L/K,V)$]
With the notation as above we have $\partial(\tilde{y})=0$ in $K_{0}(\tilde\Lambda,\widehat{\mathbb{Q}^{ur}_{p}})$.
\end{conjecture}

The complex $R\Gamma(L,T)$ is a perfect complex of  $\Lambda$-modules and $\mathrm{Ind}_{L/\mathbb{Q}_{p}} T$ is a f.g.\ projective $\Lambda$-module, thus $M^{\bullet}:=R\Gamma(L,T)\oplus \mathrm{Ind}_{L/\mathbb{Q}_{p}} T [0]$ is a perfect complex of  $\Lambda$-modules with $\mathrm{H}^{0}(M^{\bullet})\cong \mathrm{Ind}_{L/\mathbb{Q}_{p}} T$, $\mathrm{H}^{1}(M^{\bullet})\cong \mathrm{H}^{1}(L,T)$, $\mathrm{H}^{2}(M^{\bullet})\cong \mathrm{H}^{2}(L,T)$ and $\mathrm{H}^{i}(M^{\bullet})=0$ for $i\geq 3$. There is an isomorphism
\begin{equation*}
\xymatrix{
comp_{V} \circ exp^{-1} :\mathrm{H}^{1}(B_{dR}\otimes M^{\bullet})\rightarrow \mathrm{H}^{0}(B_{dR}\otimes M^{\bullet}),
} \end{equation*}
and we define $C_{L/K}:=\chi(M^{\bullet}, comp_{V} \circ exp^{-1})\in K_{0}(\tilde{\Lambda}, B_{dR})$ to be the refined Euler characteristic (see \cite[2.6]{Breu}).

Set
\begin{equation*}
U_{cris}:=\partial([\mathrm{D}^{L}_{cris}(V),1-\phi])+\partial([\mathrm{D}^{L}_{cris}(V^{*}(1))^{*},(1-\phi^{*})^{-1}])\in K_{0}(\Lambda,\mathbb{Q}_{p}).
\end{equation*}
Note that
\begin{equation*}
\partial([\mathrm{D}^{L}_{cris}(V^{*}(1))^{*},(1-\phi^{*})^{-1}])=-\partial([\mathrm{D}^{L}_{cris}(V^{*}(1)),1-\phi]),
\end{equation*}
since $(-)^{-1}$ induces multiplication with $-1$ on relative $K_{0}$-groups, so that
\begin{equation*}
U_{cris}=\partial([\mathrm{D}^{L}_{cris}(V),1-\phi])-\partial([\mathrm{D}^{L}_{cris}(V^{*}(1)),1-\phi]).
\end{equation*}
Finally, the multiplication of $\tilde{\alpha}_{V,L/K}$ with $t $ and the equivariant $\epsilon$-factor translates in the language of relative $K_{0}$-groups into the summation of their images under $\partial$.

Consider the class
\begin{equation*}
C_{L/K}+\partial(t)+\partial(\epsilon_{D}(L/K,V)).
\end{equation*}
This belongs to $K_{0}(\tilde{\Lambda}, \widehat{\mathbb{Q}^{ur}_{p}})$, because it is invariant under the action of $G_{\mathbb{Q}^{ur}_{p}}$ and $K_{0}(\tilde{\Lambda}, \widehat{\mathbb{Q}^{ur}_{p}})=K_{0}(\tilde{\Lambda}, B_{dR})^{G_{\mathbb{Q}^{ur}_{p}}}$. The conjecture $C^{na}_{ep}(L/K,V)$ takes the form:
\begin{equation} \label{eq:82}
C_{L/K}+U_{cris}
+\partial(t)+\partial(\epsilon_{D}(L/K,V))=0
\end{equation}
in $K_{0}(\tilde{\Lambda}, \widehat{\mathbb{Q}^{ur}_{p}})$. 

\section{Tamely ramified extension}

Let $L/K$ be a tamely ramified extension, then $\mathcal{O}_{L}$ is a f.g.\ projective $\Lambda$-module (see \cite[Cor.\ 1]{Fr}). We fix an element $t^{ur}\in (\widehat{\mathbb{Z}^{ur}_{p}}^{\times}(\chi^{ur}_{\mathbb{Q}_{p}}))^{Gal(\overline{\mathbb{Q}_{p}}/\mathbb{Q}_{p})}$. Note that two such elements differ by an element of $\mathbb{Z}_{p}^{\times}$. Then $\mathcal{O}_{L}$ and $(\mathcal{O}_{\widehat{L^{ur}}}(\chi^{ur}))^{G(L^{ur}/L)}$ are isomorphic as $\Lambda$-modules by just sending $l\in \mathcal{O}_{L}$ to $t^{ur}\cdot l \in (\mathcal{O}_{\widehat{L^{ur}}}(\chi^{ur}))^{G(L^{ur}/L)}$. Thus every element $\tilde{l}\in (\mathcal{O}_{\widehat{L^{ur}}}(\chi^{ur}))^{G(L^{ur}/L)}$ can be written as
\begin{equation*}
\tilde{l}=t^{ur}\cdot l=t^{ur}\cdot \sum_{g\in G} a_{g}  g(b), \quad a_{g}\in \mathcal{O}_{K}
\end{equation*}
for a normal integral basis $b$ of $\mathcal{O}_{L}$ over $\mathcal{O}_{K}$ \cite[Cor.\ 1]{Fr}. Obviously, the same is true for $L$ and $(\widehat{L^{ur}}(\chi^{ur}))^{G(L^{ur}/L)}$.

Denote by $v$ a basis of $\mathbb{Z}_{p}(\chi^{ur}_{\mathbb{Q}_{p}})$, then the element $v\otimes \xi$ substitutes a basis of $T$. Moreover,
$\mathrm{D}_{dR}(\mathrm{Ind}_{L/\mathbb{Q}_{p}}V)\cong\mathrm{D}^{L}_{dR}(V)$ (resp.  $\mathrm{D}_{dR}(\mathrm{Ind}_{L/\mathbb{Q}_{p}}V(-1))\cong\mathrm{D}^{L}_{dR}(V(-1))$ is a one-dimensional $L$-vector space with the basis $e_{\chi^{ur}_{\mathbb{Q}_{p}},1}:=t^{ur}\cdot t^{-1}\otimes (v\otimes \xi)$  (resp. $e_{\chi^{ur}_{\mathbb{Q}_{p}},0}:=t^{ur}\otimes v$). In particular, they are isomorphic as $\Omega$-modules and we have a commutative diagram of $B_{dR}[G]$-modules (with an action of $G_{\mathbb{Q}_{p}}$)

\begin{equation} \label{eq:116}
\xymatrix{
B_{dR}\otimes_{\mathbb{Q}_{p}} \mathrm{D}^{L}_{dR}(V) \ar[rr]^{comp_{V}}_{\cong} \ar[d] && B_{dR}\otimes_{\mathbb{Q}_{p}} \mathrm{Ind}_{L/\mathbb{Q}_{p}}V \ar[d]^{t\cdot\otimes f}\\
B_{dR}\otimes_{\mathbb{Q}_{p}} \mathrm{D}^{L}_{dR}(V(-1)) \ar[rr]^{comp_{V(-1)}}_{\cong} && B_{dR}\otimes_{\mathbb{Q}_{p}} \mathrm{Ind}_{L/\mathbb{Q}_{p}}V(-1),\\
} \end{equation}
where the map $t\cdot$ is the multiplication with $t$ and $f(v\otimes \xi)=v$.
\begin{flushleft}
\textbf{Warning:} the left vertical arrow in the above diagram is an isomorphism of $\Omega$-modules induced by $e_{\chi^{ur}_{\mathbb{Q}_{p}},1}\mapsto e_{\chi^{ur}_{\mathbb{Q}_{p}},0}$, whereas the right vertical arrow is an isomorphism of $B_{dR}[G]$-modules with an action of $G_{\mathbb{Q}_{p}}$ and is responsible for the later normalization on $K_{1}$-groups.
\end{flushleft}

Set $K^{\bullet}:=R\Gamma(L,T)\oplus \mathcal{O}_{L}e_{\chi^{ur}_{\mathbb{Q}_{p}},1}[0]$, a perfect complex of  $\Lambda$-modules with
\begin{align*}
\mathrm{H}^{0}(K^{\bullet})\cong \mathcal{O}_{L}e_{\chi^{ur}_{\mathbb{Q}_{p}},1}\text{, }\quad\mathrm{H}^{1}(K^{\bullet})\cong \mathrm{H}^{1}(L,T){\color{black}\text{, }\quad\mathrm{H}^{2}(K^{\bullet})\cong \mathrm{H}^{2}(L,T)}\\ \quad\text{and}\quad\mathrm{H}^{i}(K^{\bullet})=0\text{ for }i\geq 2.\phantom{mmmmmmmmmmmmmmmmmmmmmm}
\end{align*}
The composition rule for the refined Euler characteristic gives the equality
\begin{equation} \label{log:eq:2}
C_{L/K}=\chi(K^{\bullet},exp^{-1})+ [\mathcal{O}_{L}e_{\chi^{ur}_{\mathbb{Q}_{p}},1},comp_{V},\mathrm{Ind}_{L/\mathbb{Q}_{p}} T]\\
\end{equation}
in $K_{0}(\tilde{\Lambda}, B_{dR})$.

Recall that $\mathcal{F}(\mathfrak{p}_{L})$ is a cohomologically trivial $\Lambda$-module (see \cite[Prop.\ 3.9]{CG}). Moreover, by \cite[Lem.\ 1.1]{BKS} $\mathcal{F}(\mathfrak{p}_{L})[-1]$ is a perfect complex of $\Lambda$-modules. We set
\begin{equation*}
E_{L/K}(\mathcal{F}(\mathfrak{p}_{L})):=\mathcal{F}(\mathfrak{p}_{L})[-1]\oplus \mathcal{O}_{L}[0],
\end{equation*}
a perfect complex of $\Lambda$-modules with
\begin{equation*}\mathrm{H}^{0}(E_{L/K}(\mathcal{F}(\mathfrak{p}_{L})))\cong \mathcal{O}_{L},\;\; \mathrm{H}^{1}(E_{L/K}(\mathcal{F}(\mathfrak{p}_{L})))\cong \mathcal{F}(\mathfrak{p}_{L})
\end{equation*}
and $\mathrm{H}^{i}(E_{L/K}(\mathcal{F}(\mathfrak{p}_{L})))=0$ for $i\geq 2$. Using the identification
\begin{equation*}
Le_{\chi^{ur}_{\mathbb{Q}_{p}},1}=\mathrm{D}^{L}_{dR}(V)=t_{V}(L)\cong \hat{\mathbb{G}}_{a}(L)= L,
\end{equation*}
the $\Lambda$-module isomorphism $\mathcal{O}_{L}e_{\chi^{ur}_{\mathbb{Q}_{p}},1}\cong \mathcal{O}_{L}$ and the diagram \eqref{log:eq:1} we get the equality
\begin{equation} \label{eq:83}
\chi((K^{\bullet}),exp^{-1}) = \chi(E_{L/K}(\mathcal{F}(\mathfrak{p}_{L})), log_{\mathcal{F}}) + { [{\mathbb{Z}_p[G/I]}, {1-\chi^{ur}(Fr_K)^{-1}Fr_K} , {\mathbb{Z}_p[G/I]}]} \;
\end{equation}
$\text{ in }K_{0}(\Lambda, \mathbb{Q}_{p})$, the maps $exp^{-1}$ and $log_{\mathcal{F}}$ being $\Omega$-module-isomorphisms. {  Here the term  $[{\mathbb{Z}_p[G/I]}, {1-\chi^{ur}(Fr_K)^{-1}Fr_K} , {\mathbb{Z}_p[G/I]}]$ represents $\mathrm{H}^2(L,T)$ by Proposition \ref{fg:pr:1} (4).}

Now let $n_{0}\in \mathbb{N}$ be big enough such that $\mathfrak{p}^{n_{0}}_{L}$ is a projective $\Lambda$-module and
\begin{equation*}
log_{\mathcal{F}}: \mathcal{F}(\mathfrak{p}^{n_{0}}_{L}) \overset{\cong}{\longrightarrow} \hat{\mathbb{G}}_{a}(\mathfrak{p}^{n_{0}}_{L})=\mathfrak{p}^{n_{0}}_{L}.
\end{equation*}
Then $\mathcal{F}(\mathfrak{p}^{n_{0}}_{L})$ is a projective $\Lambda$-submodule of finite index in
$\mathcal{F}(\mathfrak{p}_{L})$, hence we can define $E_{L/K}(\mathcal{F}(\mathfrak{p}^{n_{0}}_{L}))$ analogously to the previous consideration. But
\begin{equation*}
[\mathcal{F}(\mathfrak{p}^{n_{0}}_{L}), log_{\mathcal{F}}, \mathfrak{p}^{n_{0}}_{L}]=0 \text{ in } K_{0}(\Lambda, \mathbb{Q}_{p}),
\end{equation*}
so that
\begin{equation} \label{log:eq:3}
\begin{array}{rcl}
\chi(E_{L/K}(\mathcal{F}(\mathfrak{p}^{n_{0}}_{L})), log_{\mathcal{F}}) & = & [\mathcal{F}(\mathfrak{p}^{n_{0}}_{L}), log_{\mathcal{F}}, \mathfrak{p}^{n_{0}}_{L}]+[\mathfrak{p}^{n_{0}}_{L},id,\mathcal{O}_{L}]\\
\\
& = & [\mathfrak{p}^{n_{0}}_{L},id,\mathcal{O}_{L}].
\end{array}
\end{equation}

The exact sequence
\begin{equation*}
\xymatrix{
0 \ar[r] & \mathcal{F}(\mathfrak{p}^{n_{0}}_{L}) \ar[r]^{s} & \mathcal{F}(\mathfrak{p}_{L}) \ar[r] & \mathcal{F}(\mathfrak{p}_{L})/\mathcal{F}(\mathfrak{p}^{n_{0}}_{L}) \ar[r] & 0
}\end{equation*}
gives rise to a distinguished triangle of perfect complexes of $\Lambda$-modules
\begin{equation*}
\xymatrix{
E_{L/K}(\mathcal{F}(\mathfrak{p}^{n_{0}}_{L})) \ar[r]^{j} & E_{L/K}(\mathcal{F}(\mathfrak{p}_{L})) \ar[r] & \mathrm{cone}(j),
}\end{equation*}
where $j=s_{*}\oplus id_{\mathcal{O}_{L}}$, such that $\mathrm{H}^{0}(\mathrm{cone}(j))=0$, $\mathrm{H}^{1}(\mathrm{cone}(j))\cong \mathcal{F}(\mathfrak{p}_{L})/\mathcal{F}(\mathfrak{p}^{n_{0}}_{L})$ and $\mathrm{H}^{i}(\mathrm{cone}(j))=0$ for $i\geq 2$. This triangle together with \eqref{log:eq:3} leads to the equalities
\begin{equation} \label{log:eq:4}
\begin{array}{rcl}
\chi(E_{L/K}(\mathcal{F}(\mathfrak{p}_{L})), log_{\mathcal{F}}) & = & \chi(E_{L/K}(\mathcal{F}(\mathfrak{p}^{n_{0}}_{L})), log_{\mathcal{F}})+\chi(\mathrm{cone}(j),0)\\
\\
& = & [\mathfrak{p}^{n_{0}}_{L},id,\mathcal{O}_{L}]+\chi(\mathcal{F}(\mathfrak{p}_{L})/\mathcal{F}(\mathfrak{p}^{n_{0}}_{L})[-1],0).\\

\end{array}
\end{equation}

The quotients $\mathcal{F}(\mathfrak{p}_{L})/\mathcal{F}(\mathfrak{p}^{n_{0}}_{L})$ and $\mathfrak{p}_{L}/\mathfrak{p}^{n_{0}}_{L}$ are filtered by the images of $\mathcal{F}(\mathfrak{p}^{i}_{L})$ and $\mathfrak{p}^{i}_{L}$ , respectively, for $i\geq 1$. The associated graded objects considered as complexes are canonically isomorphic perfect complexes of $\Lambda$-modules, thus by \cite[Prop.\ 2.1(iii)]{BlB} we have the equality
\begin{equation} \label{eq:78}
\begin{array}{rcl}
\chi(\mathcal{F}(\mathfrak{p}_{L})/\mathcal{F}(\mathfrak{p}^{n_{0}}_{L})[-1],0) & = & \chi( \mathfrak{p}_{L}/\mathfrak{p}^{n_{0}}_{L}[-1],0)\\
\\
& = & [\mathfrak{p}_{L},id,\mathfrak{p}^{n_{0}}_{L}]\\
\\
& = & [\mathcal{O}_{L},id,\mathfrak{p}^{n_{0}}_{L}]-[\mathcal{O}_{L},id,\mathfrak{p}_{L}].
\end{array}
\end{equation}

Let $q_{K}=p^{f}:=[\mathcal{O}_{K}:\mathfrak{p}_{K}]$ and $e_{I}:=\frac{1}{\left|I\right|}\sum_{i\in I}i$ be the idempotent of $\Omega$ associated to the inertia subgroup $I$ of $G$. Let $^{\sharp}x\in K_{1}(\Omega)\subset K_{1}(\tilde{\Omega})$ be defined for every element $x\in Cent(\Omega)$ as follows. If $Cent(\Omega)=\prod F_{i}$ is the Wedderburn decomposition of $Cent(\Omega)$ into a product of fields and $x=(x_{i})$ under this decomposition, then $^{\sharp}x=(^{\sharp}x_{i})$ with $^{\sharp}x_{i}=x_{i}$ if $x_{i}\neq 0$ and $^{\sharp}x_{i}=1$ if $x_{i}=0$.

The normal basis theorem for $\mathcal{O}_{L}/\mathfrak{p}_{L}$ over $\mathbb{Z}_{p}/p\mathbb{Z}_{p}$ implies that there exists a short exact sequence of $G/I$-modules
\begin{equation*}
\xymatrix{
0 \ar[r] & p\cdot \mathbb{Z}_{p}[G/I]^{f} \ar[r] & \mathbb{Z}_{p}[G/I]^{f} \ar[r] & \mathcal{O}_{L}/\mathfrak{p}_{L} \ar[r] & 0.
}\end{equation*}
Using this sequence we compute that
\begin{equation} \label{log:eq:5}
[\mathcal{O}_{L},id,\mathfrak{p}_{L}]= -\partial(^{\sharp}(q_{K} e_{I})).
\end{equation}
Observing \eqref{eq:78} and \eqref{log:eq:5}  the equality \eqref{log:eq:4} becomes
\begin{equation} \label{log:eq:6}
\begin{array}{rcl}
\chi(E_{L/K}(\mathcal{F}(\mathfrak{p}_{L})), log_{\mathcal{F}}) & = &
[\mathfrak{p}^{n_{0}}_{L},id,\mathcal{O}_{L}]+[\mathcal{O}_{L},id,\mathfrak{p}^{n_{0}}_{L}]+\partial(^{\sharp}(q_{K} e_{I}))\\
\\
& = &  \partial(^{\sharp}(q_{K} e_{I})).\\
\end{array}
\end{equation}

Write $\Sigma(L)$ for the set of all embeddings $L\rightarrow \overline{\mathbb{Q}_{p}}$ fixing $\mathbb{Q}_{p}$. For each $\sigma \in \Sigma(K)$ we fix $\hat{\sigma}\in \Sigma(L)$ such that $\hat{\sigma}|_{K}=\sigma$. Let $b\in \mathcal{O}_{L}$ be a $K[G]$-basis of $L$ and let $\chi$ be an irreducible $\overline{\mathbb{Q}_{p}}$-valued character of $G$. The norm resolvent is defined by
\begin{equation*}
\mathcal{N}_{K/\mathbb{Q}_{p}}(b|\chi):=\prod_{\sigma\in \Sigma(K)} \mathrm{Det}_{\chi}(\sum_{g\in G}\hat{\sigma}(g(b))g^{-1})\in \overline{\mathbb{Q}_{p}}^{\times},
\end{equation*}
where $\mathrm{Det}_{\chi}$ is the homomorphism $\overline{\mathbb{Q}_{p}}[G]^{\times}\rightarrow \overline{\mathbb{Q}_{p}}^{\times}$ given by
\begin{equation*}
\mathrm{Det}_{\chi}(\sum_{g\in G}a_{g}g):=\det(\sum_{g\in G}a_{g}\rho_{\chi}(g))
\end{equation*}
and $\rho_{\chi}:G\rightarrow GL_{\chi(1)}(\overline{\mathbb{Q}_{p}})$ is a matrix representation with character $\chi$.
Note that the definition of $\mathcal{N}_{K/\mathbb{Q}_{p}}(b|\chi)$ depends on the choice of the $\hat{\sigma}$. We also let $\left\{ a_{\sigma}:\sigma\in \Sigma(K)\right\}$ be a fixed $\mathbb{Z}_{p}$-basis of $\mathcal{O}_{K}$ and define
\begin{equation*}
\delta_{K}:=det((\eta(a_{\sigma}))_{\eta,\sigma\in \Sigma(K)})\in \overline{\mathbb{Q}_{p}}^{\times}.
\end{equation*}
This is a square root of the discriminant of $K$ and depends on the choice of the $a_{\sigma}$.

\begin{lemma} \label{le:79}
There is an equality
\begin{equation*}
[\mathcal{O}_{L}e_{\chi^{ur}_{\mathbb{Q}_{p}},1},comp_{V},\mathrm{Ind}_{L/\mathbb{Q}_{p}} T]+\partial(t)=\partial(\theta) \text{ in } K_{0}(\tilde{\Lambda},\widehat{L^{ur}}),
\end{equation*}
where $\theta=(\theta_{\chi})_{\chi\in Irr(G)}\in K_{1}(\overline{\mathbb{Q}_{p}}[G])$ with { $\theta_{\chi}=\delta_{K}^{\chi(1)}\mathcal{N}_{K/\mathbb{Q}_{p}}(b|\chi)$.}
\end{lemma}

\begin{proof}
First we remark that the class
\begin{equation*}
[\mathcal{O}_{L}e_{\chi^{ur}_{\mathbb{Q}_{p}},1},comp_{V},\mathrm{Ind}_{L/\mathbb{Q}_{p}} T]+\partial(t)
\end{equation*}
is invariant under the action of $G_{L^{ur}}$, hence we can consider it in $K_{0}(\tilde{\Lambda},\widehat{L^{ur}})$.

The unramified representation $V(-1)$ is $\mathbb{C}_{p}$-admissible (see \cite[Prop.\ 3.56]{FO}), thus we may replace the ring $B_{dR}$ by $\mathbb{C}_{p}$ in the definition of the comparison isomorphism getting
\begin{equation*}
\begin{array}{rcl}
comp_{V(-1),L/\mathbb{Q}_{p}}:\mathbb{C}_{p}\otimes_{\mathbb{Q}_{p}}\mathrm{D}_{dR}(\mathrm{Ind}_{L/\mathbb{Q}_{p}}V(-1))& \overset{\cong}{\rightarrow} & \mathbb{C}_{p}\otimes_{\mathbb{Q}_{p}} \mathrm{Ind}_{L/\mathbb{Q}_{p}}V(-1),\\
\\
c\otimes x & \mapsto & c x
\end{array}
\end{equation*}
a $\mathbb{C}_{p}$-linear map, which commutes with the action of $G_{\mathbb{Q}_{p}}$. Taking invariants under $G(\overline{\mathbb{Q}_{p}}/L^{ur})$ on both sides and using the theorem of Ax-Sen-Tate the isomorphism above becomes
\begin{equation*}
comp_{V(-1),L/\mathbb{Q}_{p}}:\widehat{L^{ur}}\otimes_{\mathbb{Q}_{p}}Le_{\chi^{ur}_{\mathbb{Q}_{p}},0} \overset{\cong}{\rightarrow} \widehat{L^{ur}}\otimes_{\mathbb{Q}_{p}}\mathrm{Ind}_{L/\mathbb{Q}_{p}}V(-1)
\end{equation*}
and is induced (via tensor product $\mathbb{Q}_{p}\otimes_{\mathbb{Z}_{p}}$) by
\begin{equation*}
\mathcal{O}_{\widehat{L^{ur}}}\otimes_{\mathbb{Z}_{p}}\mathcal{O}_{L}e_{\chi^{ur}_{\mathbb{Q}_{p}},0}\overset{\cong}{\rightarrow} \mathcal{O}_{\widehat{L^{ur}}}\otimes_{\mathbb{Z}_{p}} \mathrm{Ind}_{L/\mathbb{Q}_{p}} T(-1).
\end{equation*}
From diagram \eqref{eq:116} we deduce that
\begin{equation} \label{eq:119}
[\mathcal{O}_{L}e_{\chi^{ur}_{\mathbb{Q}_{p}},1},comp_{V},\mathrm{Ind}_{L/\mathbb{Q}_{p}} T]+\partial(t)=[\mathcal{O}_{L}e_{\chi^{ur}_{\mathbb{Q}_{p}},0},comp_{V(-1)},\mathrm{Ind}_{L/\mathbb{Q}_{p}} T(-1)]
\end{equation}
in $K_{0}(\tilde{\Lambda},\widehat{L^{ur}})\subseteq K_{0}(\tilde{\Lambda},\mathbb{C}_{p})$, whence to prove the lemma we have to compute the last class.

Let $V_{triv}\cong \mathbb{Q}_{p}$ denote the trivial representation of $G_{K}$ and let $V(-1)=\qp v$. Fix a set $R$ of representatives of $G_{\qp}|G_K$ (consisting of the $\hat{\sigma}$ as chosen above). There is a commutative diagram of $\mathbb{C}_{p}[G]$-modules (with an action of $G_{\mathbb{Q}_{p}}$)
\begin{equation*}
\xymatrix{
\mathbb{C}_{p}\otimes_{\mathbb{Q}_{p}} \mathrm{D}^{L}_{dR}(V_{triv}) \ar[rr]^{comp_{V_{triv}}}_{\cong} \ar[d]^{f_{1}} && \mathbb{C}_{p}\otimes_{\mathbb{Q}_{p}} \mathrm{Ind}_{L/\mathbb{Q}_{p}}V_{triv} \ar[d]^{f_{2}}\\
\mathbb{C}_{p}\otimes_{\mathbb{Q}_{p}} \mathrm{D}^{L}_{dR}(V(-1)) \ar[rr]^{comp_{V(-1)}}_{\cong} && \mathbb{C}_{p}\otimes_{\mathbb{Q}_{p}} \mathrm{Ind}_{L/\mathbb{Q}_{p}}V(-1),\\
} \end{equation*}
where we use the following identifications
\begin{equation*}
\begin{array}{lcl}
\mathbb{C}_{p}\otimes_{\mathbb{Q}_{p}}\mathrm{Ind}_{L/\mathbb{Q}_{p}}V_{triv}\cong  {\color{green} \zp[G_{\qp}]\otimes_{\zp[G_L]}\mathbb{C}_{p}[G]^\sharp,} \quad &&   \quad \mathrm{D}^{L}_{dR}(V_{triv})\cong L,\\
\\
\mathbb{C}_{p}\otimes_{\mathbb{Q}_{p}}\mathrm{Ind}_{L/\mathbb{Q}_{p}}V(-1)\cong {\color{green}\zp[G_{\qp}]\otimes_{\zp[G_L]}\left(\mathbb{C}_{p}[G]^\sharp\otimes \qp v\right),} \quad &&   \quad \mathrm{D}^{L}_{dR}(V(-1))\cong Le_{\chi^{ur}_{\mathbb{Q}_{p}},0},
\end{array}
\end{equation*}
for which the maps are given by the formulas
\begin{equation*}
\begin{array}{l}
  {\color{green} comp_{V_{triv}}(q\otimes l)=  \sum_{\tau\in R} \left(\tau \otimes\sum_{g\in G }q \tau g(l)g^{-1}\right)  ,}\\
{\color{green}comp_{V(-1)}(q\otimes l e_{\chi^{ur}_{\mathbb{Q}_{p}},0})=\sum_{\tau\in R} \left(\tau \otimes\sum_{g\in G }qt^{ur} \tau g(l)g^{-1}\right),}
\\
{\color{green}f_{1}\big(q\otimes l\big)=q\otimes l e_{\chi^{ur}_{\mathbb{Q}_{p}},0},
}
\\
{\color{green}f_{2}\big(\tau\otimes w\big)=\tau\otimes (t^{ur} w \otimes v)  ,\quad \tau \in  R,\; w \in \mathbb{C}_{p}[G]^\sharp.
}
\end{array}
\end{equation*}
It follows, that
\small{
\begin{equation} \label{eq:98}
\begin{array}{l}
[\mathcal{O}_{L},comp_{V_{triv}},\mathrm{Ind}_{L/\mathbb{Q}_{p}} T_{triv}]=[\mathcal{O}_{L}e_{\chi^{ur}_{\mathbb{Q}_{p}},0},comp_{V(-1)},\mathrm{Ind}_{L/\mathbb{Q}_{p}} T(-1)]\\
\\
+[\mathcal{O}_{L},f_{1},\mathcal{O}_{L}e_{\chi^{ur}_{\mathbb{Q}_{p}},0}]+[\mathrm{Ind}_{L/\mathbb{Q}_{p}} T,f_{2}^{-1},\mathrm{Ind}_{L/\mathbb{Q}_{p}} T_{triv}]
\end{array}
\end{equation}}\normalsize
in $K_{0}(\Lambda,\mathbb{C}_{p})$. But the images of the last two classes in $K_{0}(\tilde{\Lambda},\mathbb{C}_{p})$ are zeros, as $f_{1}$ and $f_{2}$ are $\tilde{\Lambda}$-module-isomorphisms. Now we are reduced to computing $[\mathcal{O}_{L},comp_{V_{triv}},\mathrm{Ind}_{L/\mathbb{Q}_{p}} T_{triv}]$. For this we set
\begin{equation*}
H_{L}:=\underset{\eta\in \Sigma(L)}{\oplus}\mathbb{Z}_{p},
\end{equation*}
which becomes
a free $\Lambda$-module under the (left) $G$-action
\[g((a_\eta)_\eta)= (a_{\eta g})_\eta.\] We consider the following commutative diagram of $\mathbb{C}_{p}[G]$-modules (with an action of $G_{\mathbb{Q}_{p}}$)
\begin{equation} \label{comp-rho}
\xymatrix{
\mathbb{C}_{p}\otimes_{\mathbb{Q}_{p}} \mathrm{D}^{L}_{dR}(V_{triv}) \ar[rr]^{comp_{V_{triv}}}_{\cong} \ar[rrd]^{\rho_{L}}_{\cong} && \mathbb{C}_{p}\otimes_{\mathbb{Q}_{p}} \mathrm{Ind}_{L/\mathbb{Q}_{p}}V_{triv} \ar[d]^{\varphi_{1}} \\
&& \mathbb{C}_{p}\otimes_{\mathbb{Z}_{p}} H_{L},
}\end{equation}
where the maps $\varphi_{1}$ and $\rho_{L}$ are given by the formulas {\color{green}
\begin{equation*}
\begin{array}{l}
\rho_{L}(q\otimes l)  =  \big(q  \eta(l) \otimes 1\big)_{\eta\in \Sigma(L)}, \quad q\in  \mathbb{C}_{p},\; l\in L;\\
\varphi_{1} \left(\sum_{\tau\in R}\big(\tau\otimes\underset{g\in G }{\sum}a_{g}^\tau g\big)\right)   = \Big(   a_{g^{-1}}^\tau\otimes 1 \Big)_{\tau g=\eta\in \Sigma(L)} \quad \text{for each } a_{g}\in\mathbb{C}_p, \tau\in  R.
\end{array}
\end{equation*}}
From the diagram \eqref{comp-rho} we deduce the equality
\begin{equation*}
[\mathcal{O}_{L},comp_{V_{triv}},\mathrm{Ind}_{L/\mathbb{Q}_{p}} T_{triv}]+[\mathrm{Ind}_{L/\mathbb{Q}_{p}} T_{triv},\varphi_{1},H_{L}]=[\mathcal{O}_{L},\rho_{L},H_{L}]
\end{equation*}
in $K_{0}(\Lambda,\mathbb{C}_{p})$. Further, \cite[Lem.\ 4.16]{Breu} says that the last class is equal to $\partial(\theta)$,  so that we have
\begin{equation}\label{eq:79}
[\mathcal{O}_{L},comp_{V_{triv}},\mathrm{Ind}_{L/\mathbb{Q}_{p}} T_{triv}]=\partial(\theta).
\end{equation}
{  because the class $[\mathrm{Ind}_{L/\mathbb{Q}_{p}} T_{triv},\varphi_{1},H_{L}]$ is obviously zero.}
\end{proof}

{
\begin{lemma} \label{le:80}
Let $L/K$ be (at most) tamely ramified. Then there exists $v\in \Lambda^{\times}$, such that $\mathrm{Det}_{\chi}(v)=\chi^{ur}_{\mathbb{Q}_{p}}\big(N_{K/\mathbb{Q}_{p}}(\frak{f}(\chi))\cdot d_{K/\mathbb{Q}_{p}}^{\chi(1)}\big)$ for all $\chi\in Irr(G)$, whence
\begin{equation*}
\partial(\epsilon_{D}(L/K,V)) = \partial( \left(\epsilon( \mathrm{Ind}_{K/\mathbb{Q}_{p}}(  \rho^{}_{\chi}),-\psi_\xi)^{-1} \right)_{\chi\in Irr(G)} ).
\end{equation*}
by Lemma \ref{lemma}  and \eqref{eq:118}.
\end{lemma}}

\begin{proof}
The character $\chi^{ur}_{\mathbb{Q}_{p}}:\mathbb{Q}_{p}^{\times}\rightarrow \mathbb{Z}_{p}^{\times}$ being a homomorphism we have
\begin{equation*}
\chi^{ur}_{\mathbb{Q}_{p}}\big(N_{K/\mathbb{Q}_{p}}(f(\chi))\cdot d_{K/\mathbb{Q}_{p}}^{\chi(1)}\big)=\chi^{ur}_{\mathbb{Q}_{p}}(N_{K/\mathbb{Q}_{p}}(\frak{f}(\chi)))\cdot \chi^{ur}_{\mathbb{Q}_{p}}(d_{K/\mathbb{Q}_{p}}^{\chi(1)}).
\end{equation*}
Let $\chi^{ur}_{\mathbb{Q}_{p}}(p)=:u\in \mathbb{Z}_{p}^{\times}$ and let $d_{K/\mathbb{Q}_{p}}=p^{m}$. Then for $u^{m}\in\mathbb{Z}_{p}^{\times}\subset \Lambda^{\times}$, $\chi\in Irr(G)$
\begin{equation*}
\chi^{ur}_{\mathbb{Q}_{p}}(d_{K/\mathbb{Q}_{p}}^{\chi(1)})=\chi^{ur}_{\mathbb{Q}_{p}}(p)^{m\cdot \chi(1)}=u^{m\cdot \chi(1)}=\det(\rho_{\chi}(u^{m}1_{G}))=\mathrm{Det}_{\chi}(u^{m}).
\end{equation*}
Recall $q_{K}=p^{f}$ and $e_{I}=\frac{1}{|I|}\sum_{i\in I}i \in \Lambda$, as $(|I|,p)=1$. Let $v':=\big(\frac{u}{u\cdot e_{I}+(1_{G}-e_{I})}\big)^{f}$. Then for $\chi\in Irr(G)$
\begin{equation*}
\begin{array}{rcl}
\mathrm{Det}_{\chi}(v') & = & \mathrm{Det}_{\chi}(u)^{f}\cdot \mathrm{Det}_{\chi}(u\cdot e_{I}+(1_{G}-e_{I}))^{-f}\\
\\
& = & u^{f\cdot \chi(1)}\cdot \det(u\cdot\rho_{\chi}(e_{I})+\rho_{\chi}(1_{G})-\rho_{\chi}(e_{I}))^{-f}.
\end{array}
\end{equation*}
There exists a basis of $V_{\rho}=V_{\rho_{\chi}}$, such that
\begin{equation*}
\begin{array}{rcl}
\rho_{\chi}(e_{I})& = & \begin{pmatrix}
1 & * & \ldots & * \\ 0 & 1 & * & \ldots & * \\ & & \ddots \\ 0 &  \ldots & 0 & 1 & * & \ldots & *\\ 0 & 0 & 0 & \ldots  & 0 & *\\ & & & & & \ddots\\ 0 & 0 & 0 & 0 & 0 & \ldots  & 0 \\
\end{pmatrix} \\
\\
\text{and} \\
\\
\rho_{\chi}(1_{G})-\rho_{\chi}(e_{I})& = & \begin{pmatrix}
0 & * & \ldots & * \\ 0 & 0 & * & \ldots & * \\ & & \ddots \\ 0 &  \ldots & 0 & 0 & * & \ldots & *\\ 0 & 0  & \ldots & 0 & 1 & * \\ & & & & & \ddots\\ 0 & 0 & 0 & 0  & \ldots & 0 & 1 \\
\end{pmatrix}, \\
\end{array}
\end{equation*}
as $e_{I}+1_{G}-e_{I}=1_{G}$. It follows, that $\det(u\cdot\rho_{\chi}(e_{I})+\rho_{\chi}(1_{G})-\rho_{\chi}(e_{I}))=u^{\mathrm{rank}(\rho_{\chi}(e_{I}))}$, hence
\begin{equation*}
\mathrm{Det}_{\chi}(v')=u^{f\cdot(\chi(1)-\mathrm{rank}(\rho_{\chi}(e_{I})))}.
\end{equation*}
But
\begin{equation*}
\chi(1)=\dim V_{\rho}, \quad \mathrm{rank}(\rho_{\chi}(e_{I}))=\dim \mathrm{Im}(\rho_{\chi}(e_{I}))=\dim V^{I}_{\rho},
\end{equation*}
so that $\chi(1)-\mathrm{rank}(\rho_{\chi}(e_{I}))=\mathrm{codim}\; V^{I}_{\rho}$. Further, since $\chi$ is a tamely ramified character, the Artin conductor $\frak{f}(\chi)$ is equal to $\mathfrak{p}_{K}^{\mathrm{codim}\; V^{I}_{\rho}}$ (see \cite[p. 22]{Ma}), whence
\begin{equation*}
\mathrm{Det}_{\chi}(v')=\chi^{ur}_{\mathbb{Q}_{p}}(p^{f\cdot \mathrm{codim}\; V^{I}_{\rho}})=\chi^{ur}_{\mathbb{Q}_{p}}(N_{K/\mathbb{Q}_{p}}(\mathfrak{p}_{K}^{\mathrm{codim}\; V^{I}_{\rho}}))=\chi^{ur}_{\mathbb{Q}_{p}}(N_{K/\mathbb{Q}_{p}}(\frak{f}(\chi))).
\end{equation*}
Now we set $v:=v'\cdot u$. It remains to prove $v\in \Lambda^{\times}$ and for this it is enough to show that $u\cdot e_{I}+(1-e_{I})\in \Lambda^{\times}$. But
\begin{equation*}
(u\cdot e_{I}+(1_{G}-e_{I}))\cdot(e_{I}+u\cdot(1_{G}-e_{I}))=u 1_{G}\in  \Lambda^{\times}.
\end{equation*}
\end{proof}

{
\begin{lemma}\label{le:81}
\begin{enumerate}
\item $\partial([\mathrm{D}^{L}_{cris}(V),1-\phi])= -\partial({^\sharp}(q_Ke_I))$
\item $\partial([\mathrm{D}^{L}_{cris}(V^{*}(1)),1-\phi])= [{\mathbb{Z}_p[G/I]}, {1-\chi^{ur}(Fr_K)^{-1}Fr_K} , {\mathbb{Z}_p[G/I]}]  $
\end{enumerate}
\end{lemma}

\begin{proof}
 (1) Let   $L^{0}$ and $K^0$ be the maximal unramified extension of $\mathbb{Q}_{p}$ contained in $L$ and $K$ respectively, $[K^{0}:\mathbb{Q}_{p}]=f_K$, $[L^{0}:K^{0}]=f_{L/K}$. Denote by $Fr_{L}$ and $Fr_{K}$ the arithmetic Frobenius of $L$ and $K$, respectively, while $\tau=Fr_{\mathbb{Q}_p}^{-1}$ denotes the absolute geometric Frobenius. Then $Fr_{K}=Fr_{\mathbb{Q}_{p}}^{f_K}$. After choosing a normal basis of $\mathcal{O}_{L^0}$ over $\mathbb{Z}_p$ we have
\begin{equation*}
\mathrm{D}^{L}_{cris}(V)=L^{0}e_{\chi^{ur}_{\mathbb{Q}_{p}},1}\cong \bigoplus_{i=0}^{f_K-1}\mathbb{Q}_p[G/I]\tau^i \;\text{ with }\; \phi(e_{\chi^{ur}_{\mathbb{Q}_{p}},1})=p^{-1}\chi^{ur}_{\mathbb{Q}_{p}}(Fr_{\mathbb{Q}_{p}}^{-1})e_{\chi^{ur}_{\mathbb{Q}_{p}},1},
\end{equation*}
i.e., by the semi-linearity of $\phi$ on the last $\mathbb{Q}_p[G/I]$-module the operator $1-\phi$ is represented by the matrix
\[\begin{pmatrix}
1 & 0 & \ldots & 0 & -AFr_K^{-1} \\
-A & 1 & 0 & \ldots & 0 \\
0 &  \ddots & \ddots & \ddots&\vdots \\
   \vdots & \ddots & -A & 1 &0\\
0 & 0 & 0 & -A & 1   \\
\end{pmatrix}\]
with respect to the basis $1,\tau,\ldots,\tau^{f_K-1}$, where $A= \frac{\chi^{ur}(Fr_{\mathbb{Q}_{p}}^{-1})}{p}$.
Thus

\begin{equation*}
\begin{array}{rcl}
\partial[\mathrm{D}^{L}_{cris}(V),1-\phi]) & = & \partial\left(^\sharp((1-p^{-f_K}\chi^{ur}(Fr_K^{-1})Fr_K^{-1})e_I)\right)\\
\\
& = & \partial\Big({^\sharp\left(\frac{1}{p^{f}}(p^{f_K}-\chi^{ur}(Fr_{K}^{-1})Fr_K^{-1})e_I\right)}\Big)\\
\\
& = & -\partial(^{\sharp}(q_{K}e_{I})),
\end{array}
\end{equation*}
since $p^{f_K}-\chi^{ur}(Fr_{K}^{-1})Fr_K^{-1}\in \mathbb{Z}_{p}[G/I]^{\times}$ as can be seen by a standard argument using the geometric series.\\
(2) Analogously,
\begin{eqnarray*}
\partial([\mathrm{D}^{L}_{cris}(V^{*}(1)),1-\phi])&=&\partial\left(^\sharp((1- \chi^{ur}(Fr_K )Fr_K^{-1})e_I)\right)\\
&=&[{\mathbb{Z}_p[G/I]}, {1-\chi^{ur}(Fr_K)^{-1}Fr_K} , {\mathbb{Z}_p[G/I]}]\\& &\phantom{mmmmm}+[{\mathbb{Z}_p[G/I]},-\chi^{ur}(Fr_K)^{-1}Fr_K,{\mathbb{Z}_p[G/I]}],
\end{eqnarray*}
but the latter class is zero being a $\mathbb{Z}_p[G]$-module-isomorphism.
\end{proof}

\begin{theorem} \label{th:100}
Let $L/K$ be a Galois extension of $p$-adic fields which is (at most) tamely ramified and let $V=\mathbb{Q}_{p}(\chi^{ur})(1)$. Then $C^{na}_{ep}(L/K,V)$ is equivalent to the vanishing of
\small{
\begin{equation} \label{eq:112}
 \partial(\theta)  +\partial( (\epsilon_{D}(L/K,V )))
\end{equation}} \normalsize
in $K_{0}(\tilde{\Lambda}, \widehat{\mathbb{Q}^{ur}_{p}})$.
\end{theorem}

\begin{proof}
The proof is given by \eqref{eq:82},  \eqref{eq:118},   \eqref{log:eq:2}, \eqref{eq:83}, \eqref{log:eq:6}, Lemmata \ref{le:79},\ref{le:80}  and  \ref{le:81}.
\end{proof}



Now by Lemmata \ref{le:80} and \ref{le:79} the equation $\eqref{eq:112}$ means that there exists $w\in K_1(\tilde{\Lambda})$ such that
\begin{equation}\label{reduction1}
\mathrm{Det}_\chi(w)= \frac{\delta_K^{\chi(1)}\mathcal{N}_{K/\mathbb{Q}_{p}}(b|\chi)}{\epsilon( \mathrm{Ind}_{K/\mathbb{Q}_{p}}(  \rho^{}_{\chi}),{\color{red}-\psi_{\xi}}) }
\end{equation}
for all $\chi\in \mathrm{Irr}_{\bar{\mathbb{Q}_p}}(G)$.

Let $\tau':=\tau_{\mathbb{Q}_{p}}(\mathrm{Ind}_{K/\mathbb{Q}_{p}}1_{K})$, where $1_{K}$ is the  trivial character of $G $.

\begin{lemma} There exists a unit $u\in \mathbb{Z}_p^\times$ and $\sigma'\in G$  such that
\label{lemma1} \[\epsilon( \mathrm{Ind}_{K/\mathbb{Q}_{p}}(  \rho^{}_{\chi}),{\color{red}-}\psi_{\xi},dx)=(u\tau' d_{K/\mathbb{Q}_p}^{-1})^{\chi(1)}\mathrm{Det}_\chi(\sigma')\epsilon(    \rho^{}_{\chi} , \psi_K,dx) \] for all $\chi$.
\end{lemma}

\begin{proof}
As Gauss sums are additive and behave inductive on degree $0$ characters we have
\[\frac{\epsilon( \mathrm{Ind}_{K/\mathbb{Q}_{p}}(  \rho^{}_{\chi}),{\color{red}-}\psi_{\xi},dx)}{\epsilon( \mathrm{Ind}_{K/\mathbb{Q}_{p}}( 1_K),{\color{red}-}\psi_{\xi},dx)^{\chi(1)}}= \frac{\epsilon(    \rho^{}_{\chi} ,{\color{red}-}\psi_{\xi}\circ Tr_{K/\mathbb{Q}_p},dx)}{\epsilon(   1_K,{\color{red}-}\psi_{\xi}\circ Tr_{K/\mathbb{Q}_p},dx)^{\chi(1)}}.\] But by \cite[(3.2.6.1)]{Ta1} $\epsilon(   1_K,{\color{red}-}\psi_{\xi}\circ Tr_{K/\mathbb{Q}_p},dx)=d_{K/\mathbb{Q}_p}$ as the measure of $\mathcal{O}_K$ is normalised to be $1$ and the conductor $n({\color{red}-}\psi_{\xi}\circ Tr_{K/\mathbb{Q}_p})$ equals the exponent of the different of $K/\mathbb{Q}_p$. Note that there exists $\sigma\in G(\bar{\mathbb{Q}}_p/\mathbb{Q}_p^{nr})$ such that $\kappa(\sigma)\psi_{\mathbb{Q}_p}=-\psi_\xi$ where $\kappa$ denotes the cyclotomic character of $G_{\mathbb{Q}_p}$. Hence, using (2) in \cite[§3.2.2]{FK} we obtain
\begin{eqnarray*}
 \epsilon( \mathrm{Ind}_{K/\mathbb{Q}_{p}}( 1_K),{\color{red}-}\psi_{\xi},dx)&=&\epsilon( {\mathrm{Ind}_{K/\mathbb{Q}_{p}}( 1_K)}|\cdot|_p^{\frac{1}{2}},{\color{red}-}\psi_{\xi},dx)\sqrt{f(\mathrm{Ind}_{K/\mathbb{Q}_{p}}( 1_K))}\\
 &=&\epsilon( {\mathrm{Ind}_{K/\mathbb{Q}_{p}}( 1_K)}|\cdot|_p^{\frac{1}{2}}, \kappa(\sigma)\psi_{\xi},dx)\sqrt{f(\mathrm{Ind}_{K/\mathbb{Q}_{p}}( 1_K))}\\
 &=&\tau_{\mathbb{Q}_{p}}(\mathrm{Ind}_{K/\mathbb{Q}_{p}}( 1_K))\det( {\sigma};\mathrm{Ind}_{K/\mathbb{Q}_{p}}( 1_K))\\
 &=&\tau'\det(\bar{\sigma};\mathrm{Ind}_{K/\mathbb{Q}_{p}}( 1_K))
\end{eqnarray*}
by  \cite[(3.4.5)]{Ta1} . Using (3.4.4) in (loc.\ cit.) we also have
\begin{eqnarray*}
\epsilon(    \rho^{}_{\chi} ,{\color{red}-}\psi_{\xi}\circ Tr_{K/\mathbb{Q}_p},dx)&=&\epsilon(    \rho^{}_{\chi} , \kappa(\sigma)\psi_{\mathbb{Q}_{p}}\circ Tr_{K/\mathbb{Q}_p},dx)\\
&=& \mathrm{Det}_\chi(\sigma')\epsilon(    \rho^{}_{\chi} ,  \psi_{\mathbb{Q}_{p}}\circ Tr_{K/\mathbb{Q}_p},dx)
\end{eqnarray*}
for some $\sigma'\in G$.
\end{proof}


Moreover, by \cite[Lem.\ 4.29]{Breu} or \cite[Lem.\ 3.7]{Breu1} $\delta_{K}/\iota(\tau')\in (\mathbb{Z}^{ur}_{p})^{\times}$, hence $\eqref{reduction1}$  is equivalent to the existence of $w'\in K_1(\tilde{\Lambda})$ such that
\begin{equation}\label{reduction4}
\mathrm{Det}_\chi(w')=  \frac{d_{K/\mathbb{Q}_p}^{\chi(1)}\mathcal{N}_{K/\mathbb{Q}_{p}}(b|\chi)}{\epsilon(    \rho^{}_{\chi} , \psi_K,dx)} = \frac{ \mathcal{N}_{K/\mathbb{Q}_{p}}(b|\chi)}{\tau_K(\rho_\chi) }
\end{equation}\label{express}
for all $\chi\in \mathrm{Irr}_{\bar{\mathbb{Q}_p}}(G)$ where the last equality holds by the following calculation
\begin{eqnarray*}
 \tau_K(\rho_\chi)&=&\epsilon( \rho^{}_{\chi}|\;|_p^{\frac{1}{2}}, \psi_K,dx_{ \psi_K}) f(\chi)^{\frac{1}{2}}\\
&=&\epsilon( \rho^{}_{\chi} , \psi_K,dx_{ \psi_K})f(\chi)^{-\frac{1}{2}} q_K^{-\chi(1)\frac{n( \psi_K)}{2}}f(\chi)^{\frac{1}{2}}\\
&=&\epsilon( \rho^{}_{\chi} , \psi_K,dx )  q_K^{-\chi(1) n(\psi_K) } \\
&=&\epsilon( \rho^{}_{\chi} , \psi_K,dx )  d_{K/\mathbb{Q}_p}^{-\chi(1)}
\end{eqnarray*}

using (3.4.5) in \cite{Ta1} for the second equality and the normalisation factor $\frac{dx_{ \psi_K}}{dx}= q_K^{-\chi(1)\frac{n( \psi_K)}{2}}$ in the third one.

%

By an old result of M.\ Taylor in Galois module theory (see \cite[(31)]{Breu}, \cite[(3.4)]{Breu1} or \cite[Thm 31]{Fr}) the term $\frac{ \mathcal{N}_{K/\mathbb{Q}_{p}}(b|\chi)}{\tau_K(\rho_\chi) } $ stems from a integral unit, because the non-ramified characteristic $^\sharp(-Fr_K e_I)$ in  (loc.\ cit.) is an integral unit itself. Upon comparing with the expression \eqref{reduction4} we now have proven the following theorem:
 }

%

\begin{theorem}[Main Theorem]
Let $K$ be a finite extension of $\mathbb{Q}_{p}$ and $L/K$ be a  {\color{red} (at most) tamely ramified} Galois extension with $G=G(L/K)$.   Let $\chi^{ur}:G_{K}\rightarrow \mathbb{Z}_{p}^{\times}$ be a continuous unramified character with $\chi^{ur}(G_{L})\neq 1$. Let $V$ be either $\mathbb{Q}_{p}(\chi^{ur})$ or $\mathbb{Q}_{p}(\chi^{ur})(1)$ -- a $p$-adic representation of $G_{K}$. Then the conjecture $C^{na}_{ep}(L/K,V)$ holds.
\end{theorem}

\begin{proof}
Only the case $V=\mathbb{Q}_{p}(\chi^{ur})$ is missing, which follows immediately from the functorial behaviour under taking Kummer dual (see \cite[Prop.\ 3.14(1)]{Iz}).
\end{proof}


\appendix

\section{Compatibilities}

\label{app}

The reduced norm plays the crucial role in the work of Fr\"{o}hlich, Taylor,  Breuning etc.\  while in the description of Fukaya and Kato \cite{FK} it does not show up at all. In order to check whether both approaches coincide (in settings where both are defined and apply) one has to check certain compatibilities, which we are going to recall in this appendix.

We fix a finite group $G$. Let $F$ a finite extension of $\mathbb{Q}_p$ over which all the absolutely irreducible representations $\mathrm{Irr}(G)$ of $G$ are defined and set $A:=F[G]$. Every such representation $\rho$ induces a $F$-algebra homomorphism $T_\rho:A\to \mathrm{End}_K(V_\rho)$ sending the central idempotent $e_\rho$ to $\mathrm{id}_{V_\rho}.$  Consider the reduced norm isomorphism
\[nr: K_1(A)\cong \prod_{\rho\in \mathrm{Irr}(G)} F^\times =Z(A)^\times\] sending a class $[P,a]$ consisting of a projective, finitely generated (left) $A$-module $P$ together with an $A$-linear automorphism $a$ to the tuple $\left(\det_F\left(\mathrm{Hom}_A(V_\rho,a)\right)\right)_\rho$. Using the following lemma we get an alternative description of this map. Let $T$ be  finitely generated $\mathbb{Z}_p$-module with a continuous $G_K$-action and consider its deformation
\[\mathbb{T}:=\mathbb{Z}_p[G]\otimes_{\mathbb{Z}_p} T,\]
a $(\mathbb{Z}_p[G] , G_K)$-bimodule, where $\mathbb{Z}_p[G]$ acts by left multiplication on the left tensor factor while the action of $\sigma\in G_K$ is induced by $\sigma (\lambda\otimes t)=\lambda \bar{\sigma}^{-1}\otimes \sigma t$, where $\bar{\sigma}$ denotes the image of $\sigma$ in $G$.

\begin{lemma}
There are canonical isomorphisms
\[\mathrm{Hom}_A(V_\rho,A\otimes_{\mathbb{Z}_p[G]}\mathbb{T})\cong V_\rho^*\otimes_{\mathbb{Z}_p}T\cong V_\rho\otimes_{\mathbb{Z}_p[G]} \mathbb{T}\]
of $F[G_K]$-modules, where the $G_K$-action is induced by the action on $\mathbb{T}$ on the outer terms, while it is diagonally in the middle term ($V_\rho^*=V_{\rho^*}$ is considered as $G_K$-module via the projection onto $G$ here).
\end{lemma}

\begin{proof}
The inverse of the first isomorphism is induced by sending $\omega\otimes t$ to the map \[\tilde{\omega}: V_\rho\to A\otimes_{\mathbb{Z}_p[G]}\mathbb{T},\;\;\;\; v\mapsto \sum_{g\in G} \omega(g^{-1}v)g\otimes t. \] Now the statement is easily checked.
\end{proof}

Hence we get immediately that
\[nr([P,a])=\left(\det_F\left(\mathrm{Hom}_A(V_\rho,a)\right)\right)_\rho=\left(\det_F\left( V_\rho\otimes_A a)\right)\right)_\rho.\] In particular, the $\rho$-component can also be determined in the way of Fukaya and Kato \cite{FK}. Furthermore we have a commutative diagram

\[\xymatrix{
  A^\times \ar[rr]^{can} \ar[dr]_{Nr}
                &  &    K_1(A) \ar[dl]^{nr}    \\
                & {\prod_{\rho\in\mathrm{Irr}(G)} F^\times}                 }\]

  where $can$ sends $a\in A$ to $[A,a]$   considering now $a$ as right multiplication by it while
  \[Nr:A^\times \to {\prod_{\rho\in\mathrm{Irr}(G)} F^\times}\] sends $a$ to $\left(\det_F(T_\rho(a))\right)_\rho$.

\bibliographystyle{amsalpha}

\end{document}